\documentclass[a4paper,10pt]{amsart}

\usepackage{ifpdf}
\usepackage{amsmath, amsthm, amsfonts, amssymb, amscd}
\usepackage[active]{srcltx}
\usepackage{color}
\usepackage{mathtools}
\usepackage[ansinew]{inputenc}

\newtheorem{theorem}{Theorem}[section] 
\newtheorem{corollary}[theorem]{Corollary}
\newtheorem{lemma}[theorem]{Lemma}

\newtheorem{proposition}[theorem]{Proposition}
\newtheorem*{proposition*}{Proposition}

\newtheorem*{question*}{Question}
\newtheorem*{theorem*}{Theorem}
\newtheorem*{claim*}{Claim}

\newtheorem*{corollary*}{Corollary}

\theoremstyle{definition}
\newtheorem{definition}[theorem]{Definition}

\theoremstyle{remark}

\newtheorem*{remark*}{Remark}

\newtheorem*{remarks*}{Remarks}

\newcommand{\R}{\mathbb{R}}\newcommand{\N}{\mathbb{N}}
\newcommand{\Z}{\mathbb{Z}}\newcommand{\Q}{\mathbb{Q}}
\newcommand{\C}{\mathbb{C}}
\newcommand{\D}{\mathbb{D}}
\renewcommand{\S}{\mathbb{S}}

\def\eqalign#1{\null\,\vcenter{\openup\jot 
\ialign{\strut\hfil$\displaystyle{##}$&
$\displaystyle{{}##}$\hfil \crcr #1\crcr }}\,}

\def\eqalignno#1{\displ@y \tabskip=\@centering
\halign to\displaywidth{\hfil$\@lign\displaystyle{##}$
\tabskip=0pt &$\@lign\displaystyle{{}##}$

\hfil\tabskip=\@centering
$\llap{$\@lign##$}\tabskip=Opt\crcr #1\crcr}}

\usepackage{graphicx}

\begin{document}

\author[P. Le Calvez]{Patrice Le Calvez}
\address{ Sorbonne Universit\'e, Universit\'e Paris-Cit\'e, CNRS, IMJ-PRG, F-75005, Paris, France \enskip \& \enskip Institut Universitaire de France}
\curraddr{}
\email{patrice.le-calvez@imj-prg.fr}

\title[A finite dimensional proof of a result of Hutchings about irrational pseudo-rotations]{A finite dimensional proof of a result of  Hutchings about irrational pseudo-rotations}

\bigskip
\bigskip

\bigskip
\begin{abstract} We prove that the Calabi invariant of a $C^1$ pseudo-rotation of the unit disk, that coincides with a rotation on the unit circle, is equal to its rotation number. This result has been shown some years ago  by Michael Hutchings (under very slightly stronger hypothesis). While the original proof used Embedded Contact Homology techniques, the  proof of this article uses generating functions and the dynamics of the induced gradient flow. 

 \end{abstract}
\maketitle



\maketitle

\bigskip
\bigskip

\bigskip
\bigskip

\section{Introduction}

\subsection{Statement of the main theorem}
We denote by $\D$ the closed unit disk of the Euclidean plane and by $\S$ the unit circle. We will furnish $\D$ with the standard area form $\omega=dx\wedge dy$ and will denote by $\mathrm {Diff}^1_{\omega}(\D)$ the group of diffeomorphisms of class $C^1$ that preserve $\omega$ (we will say that $f$ is {\it symplectic}). We will also denote by $\mathrm {Homeo}_{*}(\D)$ the group of orientation preserving homeomorphisms of $\D$.

If $f\in \mathrm {Diff}^1_{\omega}(\D)$ fixes every point in a neighborhood of $\S$, its  {\it Calabi invariant} $\mathrm{Cal}(f)\in\R$ is a well-studied object that has several interpretations (see \cite{Ca}, \cite{Fa}, \cite{ GaGh} for instance). It admits a natural extension to $\mathrm {Diff}^1_{\omega}(\D)$. We will explain this extension, as described by Benoit Joly in his thesis \cite{J}. 
If $\kappa$ is a primitive of $\omega$, and $A:\D\to\R$ a primitive of $f^*(\kappa)-\kappa$, then $\int_{\S} A \, d\mu$ does not depend on $\mu$, where $\mu$ is a Borel probability measure invariant by $f_{\vert \S}$\footnote{Such a measure is unique when the rotation number of $f_{\vert \S}$ is irrational and supported in the periodic point set otherwise.}.  Consequently, there exists a unique primitive $A_{f,\kappa}$ of $f^*(\kappa)-\kappa$ satisfying $\int_{\S} A_{f,\kappa}\, d\mu=0$ for such $\mu$. Moreover, the quantity $\mathrm{Cal}(f)=\int_{\D} A_{f,\kappa}\,\omega$ does not depend on the choice of $\kappa$.  For example we have $\mathrm{Cal}(f)=0$ if $f$ is an Euclidean rotation. Note that $\mathrm{Cal}(f)$ is the usual Calabi invariant in case $f$ fixes every point in a neighborhood of $\S$ (in that case $A_{f,\kappa}$ vanishes on a neighborhood of $\S$). 

We can also define a real function $\widetilde{\mathrm{Cal}}$ on the universal covering space $\widetilde{\mathrm {Diff}^1_{\omega}(\D)}$ of  $\mathrm {Diff}^1_{\omega}(\D)$ as follows. Every identity isotopy of $f$ in $\mathrm {Diff}^1_{\omega}(\D)$, meaning every continuous path $I=(f_s)_{s\in[0,1]}$ in $\mathrm {Diff}^1_{\omega}(\D)$ joining the identity map of $\D$ to $f$, is homotopic, relative to its endpoints, to a Hamiltonian isotopy $I'=(f'_s)_{s\in[0,1]}$. It means that there exists a time dependent divergence free vector field $(X_s)_{s\in[0,1]}$ of class $C^1$ such that for every $z\in\D$ it holds that ${d\over ds} f'_s(z)= X_s(f'_s(z))$. There exists a (uniquely defined) family $(H_s)_{s\in[0,1]}$ of functions of class $C^2$ vanishing on $\S$, such that for every $s\in [0,1]$, every $z\in\D$ and every $v\in\R^2$ it holds that  $dH_s(z).v=\omega(v, X_s(z))$. The quantity $\widetilde{\mathrm{Cal}}(I)=\int_0^1 \left(\int_{\D}  H_s\,\omega\right) \, ds$ depends only on the homotopy class $[I]$ of $I$ and we can also denote it $\widetilde{\mathrm{Cal}}([I])$. Let us give a simple example. For every $\alpha\in\R$, denote $R_{\alpha}$ the rotation of angle $2\pi\alpha$ and  $T_{\alpha}$ the isotopy $(R_{s\alpha})_{s\in[0,1]}$. It holds that $ \widetilde{ \mathrm{Cal}}([T_{\alpha}])= \pi^2\alpha$. 
As stated in \cite{J}, there is a link between $\mathrm{Cal}$ and $\widetilde{\mathrm{Cal}}$ that we will explain now. Every identity isotopy $I=(f_s)_{s\in[0,1]}$ of $f$ in $\mathrm {Diff}^1_{\omega}(\D)$ defines by restriction  an identity isotopy $I_{\vert \S}=(f_s{}_{\vert \S})_{s\in[0,1]}$  of $f_{\vert \S}$ in $\mathrm {Diff}^1(\S)$. Moreover, every homotopy class $[I]\in\widetilde{\mathrm {Diff}^1_{\omega}(\D)}$ defines by restriction a homotopy class $[I]_{\S}\in \widetilde{\mathrm {Diff}_*^1(\S)}$, where  
$\widetilde{\mathrm {Diff}_*^1(\S)}$ is the universal covering space of the group $\mathrm {Diff}_*^1(\S)$ of orientation preserving diffeomorphisms of the circle. We have the following equation, 
$$ \widetilde{ \mathrm{Cal}}([I])= \mathrm{Cal} (f)+\mathrm{rot}([I]_{\S}),$$ where $\mathrm{rot}([I]_{\S})$ is the {\it Poincar\'e rotation number} of $[I]_{\S}$. 
In particular, if $\mu$ is a Borel probability measure invariant by $f_{\Sigma}$, then the asymptotic cycle of $\mu$ defined by an isotopy $I$ (see Schwartzman \cite{Sc}) is equal to $\rho([I])\varpi\in H_1(\S,\R)$, where $\varpi$ is the fundamental class of $\S$ defining its usual orientation. 

There is a more dynamical interpretation of the Calabi invariant in the case of a compactly supported symplectic diffeomorphism of the open disk, due to Fathi \cite{Fa} and developed by Gambaudo and Ghys \cite{GaGh}. This interpretation is still valid in this more general situation. Consider the usual angular form $$d\theta ={ydx-xdy\over x^2+y^2}$$ and define
$$W=\{(z,z')\in \D\times \D\, \vert z\not=z'\}.$$  Consider $f\in  \mathrm {Diff}^1_{\omega}(\D)$. For every identity isotopy  $I=(f_t)_{t\in[0,1]}$ of $f$  in $\mathrm {Homeo}_*(\D)$ and every $(z,z')\in W$, one gets a path
$$\begin{aligned} I_{z,z'}: [0,1]&\to \R^2\\
 s&\mapsto f_s(z')-f_s(z).\end{aligned}$$
The function
 $$\begin{aligned} \mathrm{ang}_I: W&\to \R\\
 (z,z')&\mapsto {1\over 2\pi} \int _{I_{z,z'}} d\theta\end{aligned}$$
depends only on the homotopy class of $I$ in $\mathrm {Homeo}_*(\D)$ and is bounded because $f$ is a diffeomorphism of class $C^1$. In particular, one naturally gets a function $\mathrm{ang}_{[I]}$ for every $[I]\in\widetilde{\mathrm {Diff}^1_{\omega}(\D)}$. 

If $I=(f_s)_{s\in[0,1])}$ and $I'=(f'_s)_{s\in[0,1]}$ are two identity isotopies in $\mathrm {Diff}^1_{\omega}(\D)$, one can define an identity isotopy $II'=(f''_s)_{s\in[0,1]}$ in $\mathrm {Diff}^1_{\omega}(\D)$ writing:
$$f''_s= \begin{cases} f'{}_{2s} &\text{ if $0\leq s\leq  1/2$,}\\  f_{2s-1}\circ f'_1&\text{ if  $1/2\leq s\leq 1$,}\end{cases}.$$ The homotopy class of $I'J'$ depends only on $[I]$ and $[I']$ and one gets a group structure on  $\widetilde{\mathrm {Diff}}^1_{\omega}(\D)$ by setting $[I][I']=[II']$. Note that for every $n\geq 1$ we have
$$\mathrm{ang}_{[I]^n}=\sum_{k=0} ^{n-1} \mathrm{ang}_{[I]}\circ (f^k\times f^k).$$
Generalizing a proof due to Shelukhin \cite{Sh}  the following equality
 $$ \widetilde{ \mathrm{Cal}}([I])=\int_{W} \mathrm{ang}_{[I]} (z,z') \,\omega(z)\, \omega(z').$$
is proved in \cite{J}.  Let us add that this interpretation of the Calabi invariant has permitted to Gambaudo and Ghys \cite{GaGh}  to prove that two elements of $ \mathrm {Diff}^1_{\omega}(\D)$ fixing every point in a neighborhood of $\S$ and conjugate in $\mathrm {Homeo}_{*}(\D)$ have the same Calabi invariant. To conclude, just note the following:
 
 \begin {itemize}
 \item the map $[I]\mapsto \widetilde {\mathrm{Cal} (I)}$ defined on $\widetilde{\mathrm {Diff}^1_{\omega}(\D)}$
is a morphism;

\item   the map $f\mapsto   {\mathrm{Cal} (f)}$ defined on $\mathrm {Diff}^1_{\omega}(\D)$is a homogeneous quasi-morphism;

\item for every $[I]$, $[I']$ in $\widetilde{\mathrm {Diff}^1_{\omega}(\D)}$, there is a unique $k\in\Z$ such that $[I']=[I][T_k]$ and we have:
$$\mathrm{ang}_{[I']} =\mathrm{ang}_{[I]}+k, \enskip\widetilde{ \mathrm{Cal}}([I']) =\widetilde{ \mathrm{Cal}}([I])+\pi^2k.$$

\item for every $[I]\in \widetilde{\mathrm {Diff}^1_{\omega}(\D)}$ and every $p$, $q$ in $\Z$ we have
$$\widetilde{ \mathrm{Cal}}([I]^q [T_p]) =\widetilde{ \mathrm{Cal}}([I])^q+\pi^2 p.$$

\end{itemize}

 We will say that  $f\in\mathrm {Diff}^1_{\omega}(\D)$ is an {\it irrational pseudo-rotation} if it fixes $0$ and does not possess any other periodic point. Using an extension of Poincar\'e-Birkhoff theorem due to Franks \cite{Fr}, we know that for every lift $[I]\in\widetilde{\mathrm {Diff}^1_{\omega}(\D)}$ of $f$, there exists $\alpha\in\R\setminus\Q$ such that the sequence of  maps $ z\mapsto n^{-1}\mathrm{ang}_{[I]^n}(0,z)$ converges uniformly to the function $z\mapsto \alpha$. We will say that $\alpha$ is the rotation number of $[I]$ and write  $\alpha=\widetilde{ \mathrm{rot}}([I])$. Of course, $\overline \alpha=\alpha+\Z\in\R/\Z$ is independent of the choice of $I$ and we say that $f$ is a pseudo-rotation of rotation number $\overline \alpha$. Note that $\overline{\alpha}$ is the Poincar\'e rotation number of $f_{\vert \S}$.
 
 The goal of the article is to prove the following:
 
 \begin{theorem}  \label{th:principal} Let $f\in\mathrm {Diff}^1_{\omega}(\D)$ be an irrational pseudo-rotation such that $f_{\vert \S}$ is $C^1$-conjugate to a rotation. If $[I]\in\widetilde{\mathrm {Diff}^1_{\omega}(\D)}$ is a lift of $f$, then we have $$\widetilde{ \mathrm{Cal}}([I])= \pi^2\,\widetilde{ \mathrm{rot}}([I])$$
 or equivalently, we have
 $ \mathrm{Cal}(f)= 0.$

 \end{theorem}

 It is not difficult to prove that a $C^1$ diffeomorphism of $\S$ can be extended to a $C^1$ symplectic diffeomorphism of $\D$. So, it is sufficient to prove the theorem in case $f$ coincides with a rotation on the unit circle. This theorem was already known and due to Michael Hutchings (with very slightly stronger hypothesis), appearing as an easy consequence of a theorem we will recall now. Suppose that $f\in\mathrm {Diff}^1_{\omega}(\D)$ coincides with a rotation $R_{\alpha}$ in a neighborhood of $\S$. Fix a primitive $\kappa$ of $\omega$ and note that $A_{f,\kappa}$ vanishes on $\S$ whether $\alpha$ is rational or irrational. If $O$ is a periodic orbit of $f$, one proves easily that 
 $\sum_{z\in O } A_{f,\kappa}(z)$ does not depend on the choice of $\kappa$. So, one can define the {\it mean action} of $O$ as being
 $$ \mathrm{act}(O)= {1\over\#O}\sum_{z\in O } A_{f,\kappa}(z).$$ One has the following (\cite{H})\footnote{The proof is stated for smooth diffeomorphisms but should be possibly extended supposing a low differentiability condition, as Michael Hutchings explained to us.}

 \begin{theorem}  \label{th:hutchings} Suppose that $f\in\mathrm {Diff}^{\infty}_{\omega}(\D)$ coincides with a rotation in a neighborhood of $\S$ and denote $\mathcal O$ the set of periodic orbits of $f$. Then it holds that:
 
 \begin{itemize} 
 \item if  ${\mathrm{Cal} (f)}<0$, then $\inf_{O\in \mathcal O} \mathrm {act}(O) \leq {\mathrm{Cal} (f)}$,
 
 \item  if  ${\mathrm{Cal} (f)}>0$, then $\sup_{O\in \mathcal O} \mathrm {act}(O) \geq {\mathrm{Cal} (f)}$.
 
 \end{itemize}
 
 \end{theorem}
 
 It is easy to prove that if  $f\in\mathrm {Diff}^1_{\omega}(\D)$ coincides with a rotation on $\S$, then we have $\mathrm {act}(\{0\}) =0$. Consequently, by Theorem  \ref{th:hutchings}, it holds that ${\mathrm{Cal} (f)}=0$ if $f$ is an irrational pseudo-rotation. It must be noticed that Hutchings' theorem has been recently improved by Pirnapasov \cite{P} (still in the smooth category) that uses a preliminary extension result. A much weaker hypothesis on the boundary is needed, even weaker than been a rotation on the boundary. A nice corollary, is the fact that $\mathrm{Cal}(f)=0$ for every smooth irrational pseudo-rotation $f$, regardless of any condition on the boundary.  
 
  Let us continue with a nice consequence of Theorem \ref{th:principal} . Let $D\subset \D$ be a closed disk that does not contain $0$ and does not meet $\S$. One constructs easily $f'\in\mathrm {Diff}^{\infty}_{\omega}(\D)$ arbitrarily close to the identity map for the $C^{\infty}$-topology, fixing every point outside $D$ and such that $\mathrm{Cal}(f')\not=0$. If $[I]$ is a lift of $f$ and $[I']$ is the lift of $f'$ such that $[I']_{\S}$ is the trivial homotopy, then it holds that $$\widetilde {\mathrm{Cal}} ([I][I'])=\widetilde{\mathrm{Cal}} ([I])+ \widetilde{\mathrm{Cal}}([I'])=\widetilde {\mathrm{rot}}([I])+ \mathrm{Cal}(f') \not=\widetilde{\mathrm{rot}}({[I]})= \widetilde{\mathrm{rot}}([I][I']).$$It implies that $f\circ f'$ is not an irrational pseudo-rotation, which means that it has at least one periodic orbit different from $\{0\}$. Such a periodic orbit must meet $D$.  
 
The proof of Hutchings is based on a reduction to a problem of contact geometry and then the applications of methods of Embedded Contact Homology theory. The previous perturbation result is true in a more general situation. Asaoka and Irie \cite{AI}, using Embedded Contact Homology as well, have proved that it  remains true provide $f\in\mathrm {Diff}^1_{\omega}(\D)$  has no periodic point inside $D$. 
Other striking applications of contact or symplectic geometry to dynamics on surfaces have appeared recently that use elaborated tools of contact or symplectic geometry (for instance see Cristofaro-Gardiner, Humil\`ere, Seyfaddini \cite{CrHS} or Cristofaro-Gardiner, Prasad, Zhang \cite{CrPZ}). The proof of Theorem \ref {th:principal} that will be given in the present article does not use Floer homology but only generating functions. Can be hope be able to find proofs of these deep results using such classical tools?

 \subsection{Idea of the proof}
Let us state first the following result of Bramham  \cite{Br}:

\begin{theorem} \label{bramham}Every irrational pseudo-rotation $f\in\mathrm {Diff}^{\infty}_{\omega}(\D)$ is the  limit, for the $C^0$ topology, of a sequence of finite order $C^{\infty}$ diffeomorphisms.

\end{theorem}

Write $\overline \alpha=\mathrm{rot}(f)$. The proof says something more precise: if $(q_n)_{n\geq 0}$ is a sequence of positive integers such that  $(q_n\overline{\alpha})_{n\geq 0}$ converges to $0$ in $\R/\Z$, then there exists a sequence  of homeomorphisms $(f_n)_{n\geq 0}$ fixing $0$ and satisfying $(f_n)^{q_n}=\mathrm{Id}$, that converges to $f$ for the $C^0$ topology (to obtain a sequence of smooth approximations one needs a simple additional argument of approximation). Theorem \ref{th:principal} would have been an easy consequence of Theorem  \ref{bramham} if the strongest following properties were true:
\begin{itemize}
\item the $f_n$ are symplectic diffeomorphisms of class $C^1$;
\item the sequence $(f_n)_{n\geq 0}$ converges to $f$ in the $C^1$ topology.
\end{itemize}
Unfortunately, there is no reason why the $f_n$ appearing in the construction of the sequence satisfy these properties. 
\medskip

The original proof of Theorem  \ref{bramham} uses pseudoholomorphic curve techniques from symplectic geometry. In \cite{L2} we succeeded to find a finite dimensional proof by using generating functions, like in Chaperon's broken geodesics method \cite{Ch}. Let us remind the ideas of this last proof. The hypothesis were slightly different, the result was stated for an irrational pseudo-rotation $f\in\mathrm {Diff}^{1}_{\omega}(\D)$ coinciding with a rotation on $\S$. This last property permits us to extend our map, also denoted $f$,  to a piecewise $C^1$ diffeomorphism of the whole plane, being an integrable polar twist map with increasing rotation number in an annulus $\{z\in\R^2\,\vert\,1\leq \vert z\vert \leq r_0\}$ and equal to an irrational rotation in  $\{z\in\R^2\,\vert \,\vert z\vert \geq r_0\}$.  Let $I$ be an identity isotopy of $f_{\vert \D}$ in $\mathrm {Diff}^1_{\omega}(\D)$ that fixes $0$. We write $\alpha=\mathrm {rot}([I]_{\S})$. One can extend $I$ to an identity isotopy $(f_t)_{t\in[0,1]}$ of $f$, also denoted $[I]$, where each $f_t$ is an integrable polar map in $\{z\in\R^2\,\vert\,1\leq \vert z\vert \leq r_0\}$ and a rotation in  $\{z\in\R^2\,\vert \, \vert z\vert \geq r_0\}$.  The circle ${\S}$ is accumulated from outside by invariant circles $S_{a/b}$, such that $[I]^b[T_{-a}]{}_{\vert S_{a/b}}\in \widetilde {\mathrm{Diff}}^1(S_{p/q})$ is the trivial homotopy. Moreover one can find a sequence $(a_n,b_n)\in \Z\times( \N\setminus\{0\})$ such that $\vert a_n -b_n\alpha\vert\leq b_n^{-1}$. The map $f$ being piecewise $C^1$, one can write it as the composition of $m$ maps ``close  to $\mathrm{Id}_{\R^2}$'', where $m\geq 1$, and then construct a $m$-periodic family $(h_i)_{i\in\Z}$ of generating real functions that are $C^1$ with Lipschitz derivatives. One knows that for every $b\geq 1$, the fixed point set of $f^b$ corresponds to the singular point set of a $A$-Lipshitz vector field $\zeta$ defined on a $2mb$-dimensional space $E_b$, furnished with a natural scalar product. This vector field is the gradient flow of a function ${\bf h}$ defined  in terms of the $h_i$, $i\in\Z$, and the constant $A$ is independent of $b$. In particular each circle $S_{a/b}\subset\R^2$ corresponds to a curve $\Sigma_{a}\subset E_b$ of singularities of $\zeta$. A fundamental result is the fact that $\Sigma_{a}$ bounds a disk $\Delta_{a}\subset E_b$ that contains the singular point corresponding to the fixed point $0$ and that is invariant by the flow and by a natural $\Z/b\Z$ action on $E_b$. Moreover the dynamics of the flow of $\zeta$ on $\Delta_{a}$ is north-south and the non trivial orbits have the same energy. This energy can be explicitly computed and is small if $a/b$ is a convergent of $\alpha$. Using the independence of the Lipshitz constant $A$, one can deduce that $\zeta$ is ``uniformly small'' on $\Delta_{a}$. The disk $\Delta_{a}$ projects homeomorphically on $D_{a/b}$ by an explicit map $q_1$. The $\Z/b\Z$ action on $\Delta_{a}$ defines, by projection on $D_{a/b}$, a homeomorphism $\widehat f$ of order $b$ that coincides with $f$ on $S_{a/b}$.  Using was has been said above, in particular the fact that $\zeta$ is small on $\Delta_{a}$, one can prove that $\widehat f$ is uniformly close to $f$ on the disk $D_{a/b}$ if $a/b$ is a convergent of $\alpha$. This is the way we prove the approximation result.

The orbits of $\zeta_{\vert \Delta_{a}}$ define a radial foliation on $\Delta_{a}$ that projects by $q_1$ onto a topological radial foliation ${\mathcal F}_1$ of $S_{a/b}$ that is invariant by $\widehat f$. It is a natural to ask wether one can compute the Calabi invariant of $f$ by using  the fact that  $\widehat f$ is arbitrarily close to $f$. The main problem is that $\widehat f$ and $\mathcal F_1$ are continuous. There is no differentiability anymore, and differentiability is crucial while dealing with Calabi invariants, regardless of the approach we choose. Nevertheless, the fact that we are dealing with Lipshitz maps close to $\mathrm{Id}$ (the maps $f_1$,\dots, $f_m$ appearing in the decomposition of $f$) permit us by to state some quantitative results. The three key-points that will used to get the proof are related to the foliation $\mathcal F_1$
:
\begin{enumerate}
\item the energy of the non trivial orbits of $\zeta_{\vert \Delta_{a}}$ is bounded by $K\vert a-b\alpha\vert$, where $K$ does not depend on $a$ and $b$;

\item the leaves of $\mathcal F_1$ are {\it Brouwer lines} of $I^bT_{-a}{}_{\vert S_{a/b}}$;

\item the {\it winding distance} between  $\mathcal F_1$ and $f^{-b}(\mathcal F_1)$ is bounded by $4mb$.

\end{enumerate}

Let us precise these three points.

 The non trivial orbits of $\zeta_{\vert \Delta_{a}}$ have the same energy. This quantity measures the area sweeped by a leaf of $\mathcal F_1$ along the isotopy $I^bT_{-a}{}_{\vert S_{a/b}}$. 
 
 By saying that the leaves of $\mathcal F_1$ are Brouwer lines, we mean that they are pushed on the left along the isotopy $I^bT_{-a}{}_{\vert S_{a/b}}$. Equivalently, we can say the following: let $\tilde{D}_{a/b}$ be the universal covering space of $D_{a/b}^*=D_{a/b}\setminus (S_{a/b}\cup\{0\})$ and $\tilde I=(\tilde f_s)_{s\in[0,1]}$ the identity isotopy on $\tilde{D}_{a/b}$ that lifts $I_{\vert D_{a/b}^*}$. Denote $\tilde{ \mathcal F_1}$ the (non singular) foliation of $\tilde{D}_{a/b}$ that lifts ${\mathcal F_1}_{\vert D^*_{a/b}}$.
Every leaf $\tilde \phi$ of $\tilde {\mathcal F_1}$ is an oriented topological line of  $\tilde{D}_{a/b}$ that separates its complement into two components, one on the left and one one the right. Let $T$ be the generator of the group of covering automorphisms, such that $T(\tilde \phi)$ is on the left of $\tilde\phi$. The assertion (2) states that $\tilde f^b\circ T^{-a}(\tilde\phi)$ is on the left of $\tilde\phi$ for every leaf $\tilde \phi$. The assertion (1) states that the area of the domain between $\tilde\phi$ and $\tilde f^b\circ T^{-a}(\tilde\phi)$, which is equal to the energy of the corresponding orbit of $\zeta$, is bounded by $K\vert a-b\alpha\vert$.

The assertion (3) says  that the foliation  $f^{-b}(\mathcal F_1)$ winds relative to  $\mathcal F_1$ no more than $4mq$. A precise definition will be given in the next section. Just say that if $\mathcal F$ is a radial foliation of class $C^1$, there exists $K>0$ such that  for every $b>0$ the winding distance between  the foliations $\mathcal F$ and $f^{-b}(\mathcal F)$ is bounded by $Kb$. The foliation that appears in our construction depends on $a$ and $b$, is not differentiable, possess interesting dynamical properties stated in (1) and (2) and nevertheless satisfies a ``differential-like'' property.

The properties (1), (2) and (3) will permit us to bound the Calabi invariant of each map $f_{\vert D_{a/b}}$. Note that we will use the third definition, as the map is only piecewise $C^1$ on $D_{a/b}$. The Calabi invariant will be small if $a/b$ is a convergent that is close to $\alpha$. By a limit process, we will prove that $\mathrm{Cal}(f_{\vert\D})=0$. 

The rest of the article is divided into three sections. 

Section 2 will be dedicated to the study of topological radial foliations defined on a disk. We will introduce different objects on the set of such foliations, in particular the winding distance between two foliations. We will see how to compute the Calabi invariant of a symplectic diffeomorphism by using such a foliation.

In section 3 we will recall the construction done in \cite{L2} and quickly explained in the introduction. In particular we will define $\mathcal F_1$ and prove the properties (1), (2) and (3) stated above. 

We will prove Theorem \ref{th:principal} in the short section 4.

I would like to thank Benoit Joly, Fr\'ed\'eric Le Roux and Sobhan Seyfaddini for so useful talks. I am particulary indepted to John Franks for the many conversations we had on this subject some years ago and would like to thank him warmly.

\section{Radial foliations and  Calabi invariant}\label{se:Calabi} 

We will introduce a formal setting to study radial foliations on a disk, based on a ``discretization of the angles'', that seems pertinent for studying topological foliations. Classical objects related to disk homeomorphisms, like rotation numbers or linking numbers, will be investigated in this formalism. One can look at  Bechara \cite{Be} for similar questions in a differentiable context.

\subsection{The space of radial foliation}
Let $D$ be an open disk of the form $D=\{z\in\C\, \vert \,\vert z\vert <r\}$, where $r>0$. We will set $D^*=D\setminus\{0\}$ and will denote $\tilde D$ the universal covering space of $D^*$. We will consider the sets
 $$W=\{(z,z')\in D^*\times D^*\,\vert\, z\not= z'\}$$ and
  $$\tilde W=\{(\tilde z, \tilde z')\in \tilde D\times\tilde D\,\vert\, \tilde z\not=\tilde z'\}.$$

A {\it radial foliation} is an oriented topological foliation on $D^*$ such that every leaf $\phi$ is a {\it ray}, meaning that it satisfies 
$$\alpha(\phi)=\{0\}, \enskip \omega(\phi)\subset \partial D=\{z\in\C\, \vert \,\vert z\vert =r\}.$$ In other words, $\phi$ tends to $0$ in the past and to $\partial D$ in the future. We denote ${\mathfrak F}$ the set of radial foliations. The group $\mathrm{Homeo}_*(D^*)$ of homeomorphisms of $D^*$ that are isotopic to the identity, acts naturally on ${\mathfrak F}$: if $\mathcal F\in\mathfrak F$ and $f\in\mathrm{Homeo}_*(D^*)$, then  the foliation $f(\mathcal F)$, whose leaves are the images by $f$ of the leaves of $\mathcal F$, is a radial foliation. Moreover $\mathrm{Homeo}_*(D^*)$ acts transitively on $\mathfrak F$.  If  $\mathrm{Homeo}_*(D^*)$ is endowed with the $C^0$ topology, meaning the compact-open topology applied to maps and the inverse maps, the stabilizer $\mathrm{Homeo}_{*,\mathcal F}(D^*)$ of $\mathcal F\in \mathfrak F$ is a closed subgroup of  $\mathrm{Homeo}_*(D^*)$. Consequently, the map
 $$\begin{aligned}\mathrm{Homeo}_*(D^*)/\mathrm{Homeo}_{*,\mathcal F}(D^*)&\to {\mathfrak F}\\ f \,\mathrm{Homeo}_{*,\mathcal F}(D^*)&\mapsto f(\mathcal F)
 \end{aligned}$$
 is bijective and ${\mathfrak F}$ can be furnished with a natural $C^0$ topology. It is the topology, induced from the quotient topology defined on $\mathrm{Homeo}_*(D^*)/\mathrm{Homeo}_{*,\mathcal F}(D^*)$ by this identification map. Note that this topology does not depend on $\mathcal F$.
 It is well known that the fondamental groups of $\mathrm{Homeo}_*(D^*)$ and $\mathrm{Homeo}_{*,\mathcal F}(D^*)$ are infinite cyclic and that the morphism $i_* :\pi_1(\mathrm{Homeo}_{*,\mathcal F}(D^*), \mathrm{Id})\to \pi_1(\mathrm{Homeo}_*(D^*), \mathrm{Id})$  induced by the inclusion map $i:  \mathrm{Homeo}_{*,\mathcal F}(D^*)\to \mathrm{Homeo}_*(D^*)$ is bijective. It implies that ${\mathfrak F}$ is simply connected. 
In the whole article, when $\mathcal F$ is a radial foliation, the notation $\tilde{\mathcal F}$ means the lift of $\mathcal F$ to the universal covering space $\tilde D$. Note that $\tilde{\mathcal F}$ is a trivial foliation. In particular, every leaf $\tilde \phi$ of $\tilde{\mathcal F}$  is an oriented line of $\tilde D$ that separates $\tilde D$ into two connected open sets, the component of  $\tilde D\setminus\tilde\phi$ lying on the left of $\tilde \phi$ will be denoted $L(\tilde \phi)$ and the component lying on the right will be denoted $R(\tilde \phi)$. We get a total order $\preceq_{\mathcal F}$ on the set of leaves of $\tilde {\mathcal F}$ as follows:
$$\tilde \phi \preceq_{\mathcal F}  \tilde \phi'\Longleftrightarrow R(\tilde\phi)\subset R(\tilde\phi').$$
We can also define a partial order $\leq_{ {\mathcal F}}$ on $\tilde D$: denote $\tilde\phi_{\tilde z}$ the leaf of $\tilde{\mathcal F}$ that contains $\tilde z$ and write $\tilde z<_{{\mathcal F}}\tilde z'$ if $\tilde z\not =\tilde z'$, if $\tilde\phi_{\tilde z}=\tilde\phi_{\tilde z'}$ and if the segment of $\tilde\phi_{\tilde z}$ beginning at $\tilde z$ and ending at $\tilde z'$ inherits the orientation of  $\tilde\phi_{\tilde z}$.

\subsection{Topological angles in the universal covering space}

Consider the natural projection
 $$\begin{aligned}\pi: \Z&\to \Z/4\Z\\
 k&\mapsto \dot k=k+4\Z.\end{aligned}$$ If $\Z/4\Z$ is endowed with the topology whose open sets are
 $$\emptyset, \{\dot 1\}, \{\dot 3\}, \{\dot 1, \dot 3\}, \{\dot 1, \dot 2, \dot 3\}, \{\dot 3, \dot 0,\dot 1\},\{\dot 0,\dot 1, \dot 2, \, \dot 3\},$$ and $\Z$ with the topology generated
 by the sets $2k+1$ and $\{2k+1, 2k, 2k+1\}$, $k\in\Z$, then $\pi$ is a covering map\footnote{This topology on $\Z$ is usually called the {\it digital line topology} on $\Z$}. Note that both sets
$\Z/4\Z$ and $\Z$ are non Hausdorff but path connected.

If $k$, $l$ are two integers, we will define
$$\lambda(k,l)=\begin{dcases}0 &\text{ if $k=l$,}\\
\#(k,l)\cap 4\Z +  (\#\{k,l\}\cap 4\Z)/2 &\text { if $k<l$,}\\
-\#(l,k)\cap 4\Z - (\#\{l,k\}\cap 4\Z)/2 &\text{ if $k>l$.}\end{dcases}$$
 Note that for every integers $k, l, m$ we have $$\lambda(k,l)+\lambda(l,m)=\lambda(k,m).$$
The quantity $\lambda(k,l)$ measures the ``algebraic intersection number'' of a continuous path $\gamma: [s_0,s_1]\to \Z/4\Z$ with $\{\dot 0\}$, if $\gamma$ is lifted to a path $\hat\gamma: [s_0,s_1]\to\Z$ joining $k$ to $l$.

For every $(\tilde z, \tilde z')\in \tilde W$, we can define a continuous function $\theta_{\tilde z, \tilde z'}: {\mathfrak F}\to \Z/4\Z$ as follows:
$$\theta_{\tilde z, \tilde z'}( \mathcal F)= \begin{cases} \dot 0 &\text{ if $\tilde z'>_{\mathcal F}\tilde z$,}\\
\dot 1 &\text{ if $\phi_{\tilde z'} \succ_{\mathcal F} \phi_{\tilde z}$,}\\
 \dot 2 &\text{ if  $\tilde z'<_{\mathcal F}\tilde z'$,}\\
 \dot  3 &\text{ if $\phi_{\tilde z'} \prec_{\mathcal F} \phi_{\tilde z}$.}\end{cases}$$

 The space $\mathfrak F$ being simply connected, the Lifting Theorem asserts that there exists a continuous function, 
$\hat\theta_{\tilde z, \tilde z'}: {\mathfrak F}\to \Z$, uniquely defined up to an additive constant in $4\Z$, such that $\pi\circ \hat\theta_{\tilde z, \tilde z'}= \theta_{\tilde z, \tilde z'}.$

Note that
\begin{itemize}

\item $\theta_{\tilde z', \tilde z}( \mathcal F) =\theta_{\tilde z, \tilde z'}( \mathcal F)+\dot 2$, for every $\mathcal F\in \mathfrak F$,

\item $\hat\theta_{\tilde z, \tilde z'}$ and $\hat\theta_{\tilde z', \tilde z}$ can be chosen such that $\hat\theta_{\tilde z', \tilde z}( \mathcal F) =\hat\theta_{\tilde z, \tilde z'}( \mathcal F)+2$ for every $\mathcal F\in \mathfrak F$,

\item  $\theta_{\tilde f(\tilde z), \tilde f(\tilde z')}( f(\mathcal F)) =\theta_{\tilde z, \tilde z'}( \mathcal F) $, for every $\mathcal F\in \mathfrak F$ and every $f\in\mathrm{Homeo}_*(D^*)$,

\item $\hat\theta_{\tilde f(\tilde z), \tilde f(\tilde z')}$ and $\hat\theta_{\tilde z, \tilde z'}$ can be chosen such that $\hat\theta_{\tilde f(\tilde z), \tilde f(\tilde z')}( f(\mathcal F)) =\hat\theta_{\tilde z, \tilde z'}( \mathcal F) $, for every $\mathcal F\in \mathfrak F$.

\end{itemize}

In particular, if $\mathcal F$ and $\mathcal F$ are two radial foliations, the numbers
$$ \hat\tau(\tilde z, \tilde z', \mathcal F, \mathcal F') =\hat\theta_{\tilde z, \tilde z'}( \mathcal F')- \hat \theta_{\tilde z, \tilde z'}( \mathcal F)$$
and
$$\lambda(\tilde z,\tilde z', \mathcal F, \mathcal F') = \lambda(\hat \theta_{\tilde z, \tilde z'}( \mathcal F), \hat \theta_{\tilde z, \tilde z'}( \mathcal F '))$$ do not depend on the choice of the lift $\hat \theta$.

Suppose that $(\tilde z, \tilde z')\in\tilde W$, that $f\in\mathrm{Homeo}_*({D^*})$,  that $\tilde f$ is a lift of $f$ and that $\mathcal F$, $\mathcal F'$, $\mathcal F''$ belong to $\mathfrak F$.  The following results are immediate:
\begin{itemize}

\item $\vert \lambda(\tilde z,\tilde z', \mathcal F, \mathcal F')\vert \leq \vert\hat\tau(\tilde z, \tilde z', \mathcal F, \mathcal F')\vert$,

\item  $\hat\tau(\tilde z', \tilde z, \mathcal F, \mathcal F') =  \hat\tau(\tilde z, \tilde z', \mathcal F, \mathcal F')$,

\item $\hat\tau(\tilde z,\tilde z', \mathcal F, \mathcal F')+\hat\tau(\tilde z,\tilde z', \mathcal F', \mathcal F'')=\hat\tau(\tilde z,\tilde z', \mathcal F, \mathcal F'')$,

\item $\lambda(\tilde z,\tilde z', \mathcal F, \mathcal F')+\lambda(\tilde z,\tilde z', \mathcal F', \mathcal F'')=\lambda(\tilde z,\tilde z', \mathcal F, \mathcal F'')$,

\item $ \hat\tau(\tilde z, \tilde z',  \mathcal F',  \mathcal F) =  -\hat\tau(\tilde z, \tilde z', \mathcal F, \mathcal F')$,

\item  $\lambda(\tilde z, \tilde z',  \mathcal F',  \mathcal F) =  -\lambda(\tilde z, \tilde z', \mathcal F, \mathcal F')$,

\item  $\hat\tau(\tilde f(\tilde z), \tilde f(\tilde z'),  f(\mathcal F),  f(\mathcal F')) =  \hat\tau(\tilde z, \tilde z', \mathcal F, \mathcal F')$,
\item  $\lambda(\tilde f(\tilde z), \tilde f(\tilde z'), f(\mathcal F),  f(\mathcal F')) =  \lambda(\tilde z, \tilde z', \mathcal F, \mathcal F')$.

\end{itemize}

\bigskip
The second assertion indicates that $(\tilde z,\tilde z')\mapsto \hat\tau(\tilde z', \tilde z, \mathcal F, \mathcal F')$ is symmetric. The next result precises the lack of symmetry of $(\tilde z,\tilde z')\mapsto \lambda(\tilde z', \tilde z, \mathcal F, \mathcal F')$. For every $(\tilde z, \tilde z')\in \tilde W$, define a function $\delta_{\tilde z, \tilde z'}: {\mathfrak F}\to \{-1/2,0, 1/2\}$ as follows:

$$\delta_{\tilde z, \tilde z'}(\mathcal F)=\begin{cases}0&\text{ if $\tilde \phi_{\tilde z'} = \tilde \phi_{\tilde z}$,}\\
1/2 &\text{ if $\tilde \phi_{\tilde z'}\succ_{\mathcal F} \phi_{\tilde z}$,}\\
  -1/2 &\text{ if $\tilde \phi_{\tilde z'} \prec_{\mathcal F}  \phi_{\tilde z}$.}\end{cases}$$
The next result express that $ \lambda(\tilde z, \tilde z', \mathcal F, \mathcal F')- \lambda(\tilde z', \tilde z, \mathcal F, \mathcal F')$ is the ``algebraic intersection number'' of a ``well-oriented'' continuous path $\gamma: [s_0,s_1]\to \Z/4\Z$  joining $\hat\theta_{\tilde z', \tilde z}( \mathcal F)$ to $\hat\theta_{\tilde z, \tilde z'}( \mathcal F')$ with $\{\dot 0\} -\{\dot 2\}$. 

\begin{lemma}\label{le:symmetry} For every  $(\tilde z, \tilde z')\in\tilde W$ and every $\mathcal F$, $\mathcal F'$ in $\mathfrak F$, it holds that
$$ \lambda(\tilde z, \tilde z', \mathcal F, \mathcal F')- \lambda(\tilde z', \tilde z, \mathcal F, \mathcal F')=\delta_{\tilde z,\tilde z'}(\mathcal F')- \delta_{\tilde z, \tilde z'}(\mathcal F) .$$
\end{lemma}

\begin{proof}Recall that  $$ \lambda(\tilde z, \tilde z', \mathcal F, \mathcal F')= \lambda(\hat \theta_{\tilde z, \tilde z'}( \mathcal F), \hat \theta_{\tilde z, \tilde z'}( \mathcal F '))$$ and that one can suppose that $$\hat\theta_{\tilde z', \tilde z}( \mathcal F) =\hat\theta_{\tilde z, \tilde z'}( \mathcal F)+2, \enskip \hat\theta_{\tilde z', \tilde z}( \mathcal F') =\hat\theta_{\tilde z, \tilde z'}( \mathcal F')+2.$$ Now observe that
$$ \lambda(\tilde z, \tilde z', \mathcal F, \mathcal F')- \lambda(\tilde z', \tilde z, \mathcal F, \mathcal F')=\begin{cases}0&\text{ if $\theta_{\tilde z, \tilde z'} ( \mathcal F) = \theta_{\tilde z, \tilde z'} (\mathcal F')$,} \\
0&\text{ if $\theta_{\tilde z, \tilde z'} (\mathcal F)= \dot 0$ and  $\theta_{\tilde z, \tilde z'} (\mathcal F')=\dot 2$,}\\
0&\text{ if $\theta_{\tilde z, \tilde z'} (\mathcal F)= \dot 2$ and  $\theta_{\tilde z, \tilde z'} (\mathcal F')=\dot 0$,}\\
-1&\text{ if $\theta_{\tilde z, \tilde z'} (\mathcal F) = \dot 1$ and  $\theta_{\tilde z, \tilde z'} (\mathcal F')=\dot 3$,} \\
1&\text{ if $\theta_{\tilde z, \tilde z'} (\mathcal F) = \dot 3$ and  $\theta_{\tilde z, \tilde z'} (\mathcal F')=\dot 1$,} \\
-1/2&\text{ if $\theta_{\tilde z, \tilde z'} (\mathcal F) = \dot 1$ and  $\theta_{\tilde z, \tilde z'} (\mathcal F')=\dot 0$,} \\
1/2&\text{ if $\theta_{\tilde z, \tilde z'} (\mathcal F) = \dot 0$ and  $\theta_{\tilde z, \tilde z'} (\mathcal F')=\dot 1$,} \\
-1/2&\text{ if $\theta_{\tilde z, \tilde z'} (\mathcal F)= \dot 0$ and  $\theta_{\tilde z, \tilde z'} (\mathcal F')=\dot 3$,} \\
1/2&\text{ if $\theta_{\tilde z, \tilde z'} (\mathcal F) = \dot 3$ and  $\theta_{\tilde z, \tilde z'} (\mathcal F')=\dot 0$,} \\
-1/2&\text{ if $\theta_{\tilde z, \tilde z'} (\mathcal F) = \dot 1$ and  $\theta_{\tilde z, \tilde z'} (\mathcal F')=\dot 2$,} \\
1/2&\text{ if $\theta_{\tilde z, \tilde z'} (\mathcal F) = \dot 2$ and  $\theta_{\tilde z, \tilde z'} (\mathcal F')=\dot 1$,} \\
-1/2&\text{ if $\theta_{\tilde z, \tilde z'} (\mathcal F)= \dot 2$ and  $\theta_{\tilde z, \tilde z'} (\mathcal F')=\dot 3$,}\\
1/2&\text{ if $\theta_{\tilde z, \tilde z'} (\mathcal F) = \dot 3$ and  $\theta_{\tilde z, \tilde z'} (\mathcal F')=\dot 2$.} \end{cases}
$$\end{proof}

The next result will be useful later:

\begin{lemma} \label{le:changeoffoliation} For every compact set $\tilde K\subset \tilde W$ and every  $\mathcal F, \mathcal F'$ in $\mathfrak F$, there exists $M>0$ such that for every $(\tilde z,\tilde z')\in \tilde K$  it holds that 
 $$ \vert \hat\tau (\tilde z,\tilde z',\mathcal F, \mathcal F')\vert \leq M.$$
\end{lemma}

\begin{proof} For every $(\tilde z,\tilde z')\in \tilde W$, there exists a neighborhood $\tilde d_{\tilde z,\tilde z'}\times \tilde d'_{\tilde z,\tilde z'}\subset \tilde W$ of $(\tilde z,\tilde z')$ where $\tilde d_{\tilde z,\tilde z'}$ and $\tilde d'_{\tilde z,\tilde z'}$ are topological closed disks. One can cover $\tilde K$ by a finite family of such neighborhoods. So it is sufficient to prove the result when $\tilde K=\tilde d\times \tilde d'$ is the product of two topological closed disks. 
The map
$$\begin{aligned} \tilde d\times \tilde d'\times {\mathfrak F}&\to \Z/4\Z\\
(\tilde z, \tilde z', \mathcal F)&\to\theta_{\tilde z, \tilde z'}(\mathcal F)\end{aligned}$$ 
being continuous and the space $\tilde d\times \tilde d'\times {\mathfrak F}$ being simply connected, one can lift this last map to a continuous map
$$\begin{aligned} \tilde d\times \tilde d'\times {\mathfrak F}&\to \Z\\
(\tilde z, \tilde z', \mathcal F)&\to\hat\theta_{\tilde z, \tilde z'}(\mathcal F).\end{aligned}$$ 
For every $\mathcal F, \mathcal F'$ in $\mathfrak F$, the function
$(\tilde z,\tilde z')\mapsto \hat\theta_{\tilde z, \tilde z'}(\mathcal F')- \hat\theta_{\tilde z, \tilde z'}(\mathcal F)= \tau (\tilde z, \tilde z', \mathcal F, \mathcal F')$ is continuous and so is bounded on the compact set $\tilde d\times \tilde d'$. This last affirmation comes from the fact that for every $l\geq 1$, the set
$\left\{ (\tilde z,\tilde z')\in \tilde d\times \tilde d'\, \vert \enskip\vert \tau (\tilde z, \tilde z', \mathcal F, \mathcal F')\vert <2l\right\}$ is an open subset of $\tilde d\times \tilde d'$ as the preimage of an open set by a continuous map. 
\end{proof}

Let us conclude with the following:

\begin{lemma}\label{le:distance}The function $d:{\mathfrak F}\times {\mathfrak F}\to \N\cup\{+\infty\}$ is an extended distance, where 
$$d(\mathcal F, \mathcal F')= \sup_{(\tilde z, \tilde z')\in \tilde W} \vert  \hat\tau(\tilde z, \tilde z', \mathcal F, {\mathcal F}') \vert.$$
Moreover the equality $d(f(\mathcal F), f(\mathcal F'))= d(\mathcal F, \mathcal F')$ holds for every $f\in\mathrm{Homeo}_*({D^*})$.\end{lemma}

\begin{proof} One proves immediately that $d$ is symmetric and  satisfies the triangular inequality. The fact that $d(\mathcal F, \mathcal F)=0$ for every $\mathcal F\in \mathfrak F$ is also obvious. It remains to prove that  $d(\mathcal F, \mathcal F')\not=0$ if $\mathcal F$ and $\mathcal F'$ are distinct radial foliations. In that case, one can find $(\tilde z, \tilde z')\in \tilde W$ such that $\tilde z$ and  $\tilde z'$ belong to the same leaf of $\widetilde{\mathcal F}$ and to different leaves of $\widetilde{\mathcal F'}$. Consequently, $\theta (\tilde z, \tilde z', \mathcal F)\in \{\dot 0, \dot 2\}$ and $\theta (\tilde z, \tilde z', \mathcal F')\in \{\dot 1, \dot 3\}$. It implies that 
  $\hat\theta (\tilde z, \tilde z', \mathcal F)\not=\hat\theta (\tilde z, \tilde z', \mathcal F')$ if  $\hat\theta$ is a lift of  $\theta$. 

The equality $d(\tilde f(\mathcal F), \tilde f(\mathcal F'))= d(\mathcal F, \mathcal F')$ if $f\in\mathrm{Homeo}_*({D^*})$ is a direct consequence of the equality $$\hat\tau(\tilde f(\tilde z), \tilde f(\tilde z'),  f(\mathcal F),  f(\mathcal F')) =  \hat\tau(\tilde z, \tilde z', \mathcal F, \mathcal F').$$  \end{proof}

  We will call $d$ the {\it winding distance} between two radial foliations.

\subsection{ Topological angles in the annulus}

Let $\mathcal F$ be a radial foliation and $f$ a homeomorphism of $D^*$ isotopic to the identity. Let $I=(f_s)_{s\in[0,1]}$ be an identity isotopy of $f$ in $\mathrm{Homeo}_*(D^*)$ and $\tilde I=(\tilde f_s)_{s\in[0,1]}$ the lifted identity isotopy to $\tilde D$.  
The function
$$s\in[0,1]\mapsto \theta_{\tilde f_s(\tilde z), \tilde f_s(\tilde z')} (\mathcal F)=\theta_{\tilde z,\tilde z'} ( f_s^{-1} (\mathcal F))\in \Z/4\Z$$ can be lifted to a function  
$$s\in[0,1]\mapsto \hat\theta_{\tilde z,\tilde z'} ( f_s^{-1} (\mathcal F))\in\Z$$ and the difference between the value at $1$ and the value at $0$ of this last map is nothing but $\hat \tau (\tilde z, \tilde z, {\mathcal F}, f^{-1}(\mathcal F))$. Of course it  depends neither on the choice of $I$, nor on the choice of the lift $\tilde I$ of $I$. Suppose now that $\tilde z$ and $\tilde z'$ project onto two different points of $D^*$. Note that if $k$ is large enough, then for every $s\in[0,1]$, it holds that
$$\theta_{\tilde f_s(\tilde z), \tilde f_s(T^k(\tilde z'))} (\mathcal F)=\dot 1, \enskip \theta_{\tilde f_s(\tilde z), \tilde f_s(T^{-k}(\tilde z'))} (\mathcal F)=\dot 3,$$ and so the functions
$$ s\mapsto \hat\theta_{ \tilde z,  T^k(\tilde z')} (f_s^{-1}(\mathcal F)), \enskip  s\mapsto \hat \theta_{\tilde z, T^{-k}(\tilde z')}(f_s^{-1}(\mathcal F))$$ are constant,  which means that 
$$\hat \tau (\tilde z, T^k(\tilde z'), \mathcal F,f^{-1}( {\mathcal F}) )=\hat \tau (\tilde z, T^{-k}(\tilde z'), \mathcal F, f^{-1}({\mathcal F})) =0.$$
Consequently, if $\mathcal F$ and $\mathcal F'$ are two radial foliations, then for every $(z,z')\in W$, one can define
$$\begin{aligned} \overline{\hat \tau} (z, z', \mathcal F, {\mathcal F}') &=\sum_{k\in\Z} \vert \hat \tau (\tilde z, T^k(\tilde z'), \mathcal F, {\mathcal F}') \vert,\\
 \hat \tau (z, z', \mathcal F, {\mathcal F}') &=\sum_{k\in\Z}\hat \tau (\tilde z, T^k(\tilde z'), \mathcal F, {\mathcal F}'), \\
\lambda(z, z', \mathcal F, {\mathcal F}') &=\sum_{k\in\Z} \lambda(\tilde z, T^k(\tilde z'), \mathcal F, {\mathcal F}') ,\end{aligned}$$
each sum being independent of the choice of the lifts $\tilde z, \tilde z'$ of $z, z'$. 

Suppose that $(z, z')\in W$, that $f\in\mathrm{Homeo}_*({D^*})$ and that $\mathcal F$, $\mathcal F'$, $\mathcal F''$ belong to $\mathfrak F$.  The following results are immediate

\begin{itemize}

\item  $\vert \hat\tau(z,  z',  \mathcal F',  \mathcal F)\vert \leq  \overline{\hat \tau}( z,  z', \mathcal F, \mathcal F')$,

\item  $\vert \lambda(z,  z',  \mathcal F',  \mathcal F)\vert \leq  \overline{\hat \tau}( z,  z', \mathcal F, \mathcal F')$,

\item  $\overline{\hat\tau}(z', z, \mathcal F, \mathcal F') = \overline{ \hat\tau}(z, z', \mathcal F, \mathcal F')$,

\item  $\hat\tau(z', z, \mathcal F, \mathcal F') =  \hat\tau(z, z', \mathcal F, \mathcal F')$,

\item  $\overline{\hat\tau}(z, z', \mathcal F', \mathcal F) =  \overline{\hat\tau}(z, z', \mathcal F, \mathcal F')$,

\item  $\hat\tau(z,  z',  \mathcal F,  \mathcal F')+\hat\tau( z, z', \mathcal F', \mathcal F'')= \hat\tau(z,  z',  \mathcal F,  \mathcal F'')$,

\item  $\lambda(z,  z',  \mathcal F,  \mathcal F')+\lambda(z,  z', \mathcal F', \mathcal F'')= \lambda(z,  z',  \mathcal F,  \mathcal F'')$,

\item  $\hat\tau(z, z', \mathcal F', \mathcal F) =  -\hat\tau(z, z', \mathcal F, \mathcal F')$,

\item  $\lambda(z,  z',  \mathcal F',  \mathcal F) =  -\lambda(z, z', \mathcal F, \mathcal F')$,

\item  $\hat\tau(f(z), f(z'),  f(\mathcal F),  f(\mathcal F')) =  \hat\tau(z, z', \mathcal F, \mathcal F')$,

\item  $\lambda( f(z), f(z'), f(\mathcal F),  f(\mathcal F')) =  \lambda(z,  z', \mathcal F, \mathcal F')$.

\end{itemize}

Note that we have the following:

\begin{lemma}\label{le:secondchangeoffoliation}  For every compact subset $ K\subset W$, for every  $\mathcal F, \mathcal F'$ in $\mathfrak F$, there exists $M>0$ such that for every $(z,  z')\in K$  it holds that 
 $$ \overline {\hat\tau} (z, z',\mathcal F, \mathcal F') \leq M.$$
\end{lemma}

\begin{proof}Like in the proof of Lemma \ref{le:changeoffoliation}, it is sufficient to study the case where $K=d\times d'$ is the product of two topological closed disks. Choose a lift $\tilde d$ of $d$ and a lift $\tilde d'$ of $d'$ in $\tilde D$. Choose $f\in\mathrm{Homeo}_*(D^*) $ such that $\mathcal F'=f^{-1}(\mathcal F)$ and an identity isotopy $(f_s)_{s\in[0,1]}$ of $f$. Lift this isotopy to an isotopy identity $(\tilde f_s)_{s\in[0,1]}$ on $\tilde D$. There exists $k_0>0$ such that if $k\geq k_0$, then for every $s\in[0,1]$, every $\tilde z\in \tilde d$ and every $\tilde z'\in \tilde d'$ it holds that
$$\theta_{\tilde f_s(\tilde z), \tilde f_s(T^k(\tilde z'))} (\mathcal F)=\dot 1, \enskip \theta_{\tilde f_s(\tilde z), \tilde f_s(T^{-k}(\tilde z'))} (\mathcal F)=\dot 3,$$ and it implies that 
$$\hat \tau (\tilde z, T^k(\tilde z'), \mathcal F, {\mathcal F}' )=\hat \tau (\tilde z, T^{-k}(\tilde z'), \mathcal F, {\mathcal F}') =0.$$
To conclude, it remains to apply Lemma \ref{le:changeoffoliation}  to the compact sets $\tilde d\times T^k(\tilde d')$, $\vert k\vert <k_0$. \end{proof}

Let us conclude this section by stating and analogous of  Lemma \ref{le:symmetry}.  Define first what is a {\it displacement function}. For every $f\in\mathrm{Homeo}_*(D^*)$, every lift $\tilde f$ of $f$ to $\tilde D$ and every ray $\phi$ we can define a function $m_{\tilde f, \phi}: D^*\to\Z$ as follows:
we choose a lift $\tilde \phi$ of $\phi$, then for every $z\in\D$, we consider the lift $\tilde z$ of $z$ such that $\tilde z\in\overline {L(\tilde\phi)}\cap R(T(\tilde\phi))$ and we denote $m_{\tilde f, \phi}(z)$ the integer $m$ such that  $\tilde f(\tilde z)\in\overline {L(T^m(\tilde\phi))}\cap R(T^{m+1}(\tilde\phi))$, noting that it does not depend on the choice of $\tilde \phi$. Observe that $m_{T^k\circ\tilde f,\phi}= m_{\tilde f,\phi}+k$, for every $k\in\Z$.  One proves easily the following result:

\begin{lemma}\label{le:changeofvariabledisplacement}  For every compact subset $K\subset D^*$ and every rays $\phi,\phi'$, there exists $M>0$ such that for every $f\in \mathrm{Homeo}_*(D^*)$ and every lift $\tilde f$ of $f$ to $\tilde D$, it holds that 
 $$ z\in K \enskip \mathrm{and}\enskip f(z)\in K \Rightarrow \vert m_{\tilde f, \phi} (z) -m_{\tilde f, \phi'} (z)\vert \leq M.$$
 \end{lemma}

The following result precises the lack of symmetry of $(z,z')\mapsto\lambda(z, z', \mathcal F,  {\mathcal F}')$.

\begin{lemma} \label{le:symmetryannulus} Fix $\mathcal F$, $\mathcal F'$ in $\mathfrak F$. Choose a leaf $\phi$ of $\mathcal F$, a lift $\tilde \phi$ of $\phi$ to $\tilde D$ and a map $f\in\mathrm{Homeo}_*(D^*)$ such that $\mathcal F'=f^{-1}( {\mathcal F})$.  For every $z\in D^*$, denote $\tilde z^*$ the unique lift of $z$ such that   $\tilde z\in\overline {L(\tilde\phi)}\cap R(T(\phi))$.  Then, for every $(z,z')\in W$, we have
$$\lambda(z, z', \mathcal F, \mathcal F')-\lambda(z', z, \mathcal F, \mathcal F') )= m_{\tilde f, \phi}(z')-m_{\tilde f, \phi}(z) +\delta
_{\widetilde {f(z)}^*, \widetilde{f(z')}^*}( \mathcal F) -\delta_{\tilde z^*, \tilde z^*}( \mathcal F).$$\end{lemma}

\begin{proof} Fix $(z,z')\in W$. During the proof, we will lighten the notations by writing
$$\tilde z=\tilde z^*,\enskip \tilde z'=\tilde z'{}^*, \enskip m= m_{\tilde f, \phi}(z), \enskip m'= m_{\tilde f, \phi}(z').$$ By using Lemma \ref{le:symmetry} we get$$\begin{aligned}{} &\enskip\enskip\enskip\enskip \enskip\enskip  \lambda(z, z', \mathcal F, f^{-1}( {\mathcal F}))-\lambda(z', z, \mathcal F, f^{-1}({\mathcal F}) )\\
&=\sum_{k\in\Z} \lambda(\tilde z, T^k(\tilde z'), \mathcal F, f^{-1}( {\mathcal F}))- \sum_{k\in\Z} \lambda(\tilde z', T^k(\tilde z), \mathcal F, f^{-1}({\mathcal F}) )\\
&=  \sum_{k\in\Z} \lambda(\tilde z, T^k(\tilde z'), \mathcal F, f^{-1}( {\mathcal F}))- \sum_{k\in\Z} \lambda(T^k(\tilde z'), \tilde z, \mathcal F, f^{-1}({\mathcal F}) )\\
&=  \sum_{k\in\Z} \lambda(\tilde z, T^k(\tilde z'), \mathcal F, f^{-1}( {\mathcal F}))- \lambda(T^k(\tilde z'), \tilde z, \mathcal F, f^{-1}({\mathcal F}) )\\
&= \sum_{k\in\Z} \delta_{\tilde z,T^k(\tilde z')}(f^{-1}(\mathcal F))- \delta_{\tilde z, T^k(\tilde z')}(\mathcal F)\\
&=  \sum_{k\in\Z} \delta_{\tilde f(\tilde z),T^k(\tilde f(\tilde z'))}(\mathcal F)- \delta_{\tilde z, T^k(\tilde z')}(\mathcal F)\\
&= \sum_{k\in\Z} \delta_{\tilde f(T^{-m}(\tilde z)), T^k(\tilde f(T^{-m}(\tilde z')))}(\mathcal F)- \delta_{\tilde z, T^k(\tilde z')}(\mathcal F)\\
&= \sum_{k\in\Z} \delta_{\tilde f(T^{-m}(\tilde z)), T^{k+m'-m}(\tilde f(T^{-m'}(\tilde z')))}(\mathcal F)- \delta_{\tilde z, T^k(\tilde z')}(\mathcal F).\end{aligned}$$

Recalling that $\tilde z=\tilde z^*$ and $\tilde z'=\tilde z'{}^*$, we have
$$ \delta_{\tilde z, T^k(\tilde z')}(\mathcal F)=\begin{cases} -1/2 &\text{ if $k<0$} \\ 
1/2&\text{ if $k>0$}\end{cases}$$
and
$$ \delta_{\tilde f(T^{-m}(\tilde z)), T^{k+m'-m}(\tilde f(T^{-m'}(\tilde z')))}(\mathcal F)=\begin{cases}  -1/2 &\text{ if $k+m'-m<0$} \\ 
1/2&\text{ if $k+m'-m>0$}.\end{cases}$$
So we deduce that

$$ \delta_{\tilde f(T^{-m}(\tilde z)), T^{k+m'-m}(\tilde f(T^{-m'}(\tilde z')))}(\mathcal F)- \delta_{\tilde z, T^k(\tilde z')}(\mathcal F)=\begin{cases}  0 &\text{ if $k<\min (0, m-m')$}\\
0&\text{ if $k>\max (0, m-m')$}\\
1 &\text{ if $m-m'<k<0$}\\
-1 &\text{ if $0<k<m-m'$}\\
-\delta_{\tilde z^*, \tilde z'^*}( \mathcal F)-1/2 &\text{ if $k=0<m-m'$}\\
-\delta_{\tilde z^*, \tilde z'^*}( \mathcal F)+1/2 &\text{ if $k=0>m-m'$}\\
\delta_{\widetilde {f(\tilde z)}^*, \widetilde{f(\tilde z')}^*}( \mathcal F)+1/2 &\text{ if $k=m-m'<0$}\\
\delta_{\widetilde {f(\tilde z)}^*, \widetilde{f(\tilde z')}^*}( \mathcal F)-1/2 &\text{ if $k=m-m'>0$}\\
\delta_{\widetilde {f(\tilde z)}^*, \widetilde{f(\tilde z')}^*}( \mathcal F) -\check\theta_{\tilde z^*, \tilde z'^*}( \mathcal F)&\text{ if $k=0=m-m'$}\end{cases}$$
and consequently, that
$$\sum_{k\in\Z} \delta_{\tilde f(T^{-m}(\tilde z)), T^{k+m'-m}(\tilde f(T^{-m'}(\tilde z')))}(\mathcal F)- \delta_{\tilde z, T^k(\tilde z')}(\mathcal F)= m'-m+  \delta_{\widetilde {f(z)}^*, \widetilde{f(z')}^*}( \mathcal F) -\delta_{\tilde z^*, \tilde z'{}^*}( \mathcal F).$$\end{proof}

\begin{remarks*} 1.  The fonction $\lambda_{f, \mathcal F} (z,z')\mapsto \lambda (z, z', \mathcal F, f^{-1}( {\mathcal F}))$ depends only on the foliations $\mathcal F$ and $f^{-1}(\mathcal F)$ (and not explicitely in $f$) and the fonction $z\mapsto m_{\tilde f, \phi} (z)$ depends only on the ray $ \phi$ and on $\tilde f$. 

\enskip 2. The equality proven in Lemma  \ref{le:symmetryannulus} can be written
$$\Lambda_{\tilde f, \mathcal F,  \phi}(z, z')-\Lambda_{\tilde f, \mathcal F,  \phi}(z', z) =\delta_{\widetilde {f(z)}^*, \widetilde{f(z')}^*}( \mathcal F)-\delta_{\tilde z^*, \tilde z^*}( \mathcal F) ,$$
where $$\Lambda_{\tilde f, \mathcal F, \phi}(z, z')=\lambda_{f, \mathcal F}(z, z')+m_{\tilde f, \phi} (z).$$\end{remarks*}

\enskip 3.  For every $k\in\Z$ it holds that
 $\Lambda_{T^k\circ \tilde f, \mathcal F, \phi}=  \Lambda_{\tilde f, \mathcal F, \phi}+k.$

\subsection{Rotation number and linking number.}

We will see how to define rotation numbers and self linking numbers within this formalism.

\begin{definition} \enskip Fix $f\in \mathrm{Homeo}_*(D^*) $ and a lift $\tilde f$ to $\tilde D$.  Say that $z\in D^*$ has a {\it rotation number} $\mathrm{rot}_{\tilde f}(z)\in\R$ if: 

\begin{enumerate}
\item there exists a compact set $K\subset D^*$ such that $\#\{n\geq 0\,\vert \, f^n(z)\in K\}=+\infty$;
\item if $\phi$ is a ray and if $K\subset D^*$ is a compact set containing $z$, then for every $\varepsilon>0$, there exists $n_0\geq 0$ such that for every $n\geq n_0$ it holds that
$$ f^n(z)\in K \Rightarrow \left\vert {1\over n}\sum_{i=0}^{n-1} m_{\tilde f, \phi} (f^i(z)) -\mathrm{rot}_{\tilde f}(z)\right \vert \leq \varepsilon.$$
\end{enumerate}
\end{definition}

\begin{remark*} Note that $\sum_{i=0}^{n-1} m_{\tilde f, \phi} (f^i(z))= m_{\tilde f^n, \phi} (z)$. Using Lemma \ref{le:changeofvariabledisplacement}, one deduces that if the second assertion is true for a ray $\phi$, it is true for every ray $\phi'$.\end{remark*}

\begin{definition} \label{de:rotationlinkingpoint}Fix $f\in \mathrm{Homeo}^*(D^*) $ and a lift $\tilde f$ to $\tilde D$.  Say that $(z,z')\in W$ {\it has a linking number} $\mathrm{link}_{\tilde f}(z,z')\in\R$ if 

\begin{enumerate}

\item the points $z$ and $z'$ have a rotation number;
\item there exists a compact set $K\subset W$ such that $\#\{n\geq 0\,\vert \, (f^n(z), f^n(z'))\in K\}=+\infty$;
\item if $\mathcal F$ is a radial foliation, if $\phi$ is leaf of $\mathcal F$ and if $K\subset W$ is a compact set containing $(z,z')$, then for every $\varepsilon>0$, there exists $n_0\geq 0$ such that for every $n\geq n_0$ it holds that
$$ (f^n(z), f^n(z'))\in K \Rightarrow \left\vert {1\over n}\sum_{i=0}^{n-1} {\Lambda}_{\tilde f, \mathcal F, \phi}(f^i(z),f^i(z')) -\mathrm{link}_{\tilde f}(z,z')\right \vert \leq \varepsilon.$$
\end{enumerate}

\end{definition}

\begin{remarks*}\enskip 1. Note that $$\sum_{i=0}^{n-1} {\Lambda}_{\tilde f, \mathcal F, \phi} (f^i(z), f^i(z'))={ \Lambda}_{\tilde f^n, \mathcal F, \phi} (z,z').$$

\enskip 2. As explained in the remark following the definition of the rotation number, we know that if the third assertion is true for a leaf $\phi$, it is true for every other leaf. Let us explain now, why if the third assertion is true for a foliation $\mathcal F\in \mathfrak F$, it is true for every other foliation $\mathcal F'\in  \mathfrak F$.\footnote{This assertion is moreless what is used by Gambaudo and Ghys to prove in \cite{GaGh} that two symplectic diffeomorphisms of $\D$ that are $C^0$ conjugate have the same Calabi invariant.}  We have 
$$\begin{aligned}{}&\enskip \left\vert \lambda(z,z', {\mathcal F}, f^{-n}({\mathcal F})) - \lambda(z,z', {\mathcal F}', f^{-n}({\mathcal F'}))\right\vert\\
= &\enskip \left \vert \lambda(z,z', {\mathcal F}, {\mathcal F}') + \lambda(z,z', {\mathcal F}', f^{-n}({\mathcal F}))- \lambda(z,z', {\mathcal F}', f^{-n}({\mathcal F'}))\right\vert\\
=&\enskip \left \vert \lambda(z,z', {\mathcal F}, {\mathcal F}') - \lambda(z,z', f^{-n}({\mathcal F}), f^{-n}({\mathcal F'}))\right\vert\\
=&\enskip \left \vert \lambda(z,z', {\mathcal F}, {\mathcal F}') - \lambda(f^n(z),f^n(z'), {\mathcal F}, {\mathcal F}')\right\vert\\
\leq & \enskip \overline{\hat \tau}(z,z', {\mathcal F}, {\mathcal F}') +\overline {\hat \tau} (f^n(z),f^n(z'),  {\mathcal F}, {\mathcal F}').\end{aligned}$$
It remains to apply Lemma \ref{le:secondchangeoffoliation} :  for every compact set $K\subset D^*$, there exists $M>0$ such that for every point $(z,z')$ satisfying  $(z,z')\in K$ and $(f^n(z), f^n(z'))\in K$, we have $ \overline{\hat \tau}(z,z', {\mathcal F}, {\mathcal F}')\leq M $ and  $ \overline{\hat \tau}(f^n(z),f^n(z'),  {\mathcal F}, {\mathcal F}')\leq M $.

\enskip 3. By the remark following Lemma \ref{le:symmetryannulus}, if $\mathrm{link}_{\tilde f}(z,z')$ exists, then $\mathrm{link}_{\tilde f}(z',z)$ exists and $\mathrm{link}_{\tilde f}(z',z)=\mathrm{link}_{\tilde f}(z,z')$. Indeed, we have
$$\left\vert { \Lambda}_{\tilde f^n, \mathcal F, \phi} (z,z')-{ \Lambda}_{\tilde f^n, \mathcal F, \phi} (z',z)\right\vert=\left\vert\delta_{\widetilde {f^n(z)}^*, \widetilde{f(^nz')}^*}( \mathcal F)-\delta_{\tilde z^*, \tilde z^*}( \mathcal F)\right\vert\leq 1.$$

\end{remarks*}

\begin{definition} \label{de:rotationlinkingmeasure}\enskip Fix $f\in \mathrm{Homeo}_*(D^*) $ and a lift $\tilde f$ to $\tilde D$.  Suppose that $f$ lets invariant a finite Borel measure $\mu$. 

\begin{enumerate}
\item Say that $\mu$ has a rotation number $\mathrm{rot}_{\tilde f}(\mu)\in\R$ if $\mu$-almost every point $z$  has a rotation number and if the function $\mathrm{rot}_{\tilde f}$ is $\mu$-integrable, and in that case set $$\mathrm{rot}_{\tilde f}(\mu)=\int_{\D^*} \mathrm{rot}_{\tilde f}\, d\mu.$$

\item Say that $\mu$ has a {\it self-linking number } $\mathrm{link}_{\tilde f}(\mu)\in\R$, if 

\begin{itemize}
\item $\mu$ is non atomic,
\item $\mu$ has a rotation number,
\item $\mu\times\mu$-almost every couple $(z, z')\in W$  has a linking number and the function $\mathrm{link}_{\widetilde f}$ is $\mu\times\mu$-integrable, and in that case set $$\mathrm{link}_{\tilde f}(\mu)=\int_{W} \mathrm{link}_{\tilde f}\, d\mu\times d\mu.$$
\end{itemize}

\end{enumerate}
\end{definition}

\begin{remarks*}\enskip 1. If there exists a ray $\phi$ such that $m_{\tilde f, \phi}$ is $\mu$-integrable, then by Birkhoff Ergodic Theorem, one knows that $\mu$ has a rotation number and  it holds that
$$\mathrm{rot}_{\tilde f}(\mu)=\int_{D^*} m_{\tilde f, \phi}\, d\mu.$$
 
\enskip 2. If there exists a radial foliation $\mathcal F$ and a leaf $\phi$ of $\mathcal F$ such that  $m_{\tilde f, \phi}$ is $\mu$-integrable and ${\lambda}_{f, \mathcal F}$  is $\mu\times\mu$-integrable, then $\mu$ has a self-linking number and  it holds that
$$\mathrm{link}_{\tilde f}(\mu)=\int_{W} {\Lambda}_{\tilde f, \mathcal F, \phi}\, d\mu\times d\mu.$$ 

\end{remarks*}

 Suppose that $f\in\mathrm{Diff}^1_{\omega}(\D)$. Define $\D^*=\D\setminus(\S\cup\{0\})$. Choose an identity isotopy $I=(f_s)_{s\in[0,1]}$ of $f$ in $\mathrm{Homeo}_{*}(\D)$ that fixes $0$ and write $\tilde f$ for the lift of $f_{\vert\D^*}$ naturally defined by the restriction of $I$ to $\D^*$. Extend the isotopy to a family $(f_s)_{s\in\R}$ such that $f_{s+1}=f_s\circ f$ for every $s\in[0,1]$. 
Denote $\mu_{\omega}$ the finite measure naturally defined by $\omega$. Consider the Euclidean radial foliation $\mathcal F_*$ on $\D^*$, whose leaves are the paths $(0,1)\ni t\mapsto te^{2i\pi\alpha}$, $\alpha\in[0,1)$. Note now that for every leaf $\phi$ of $\mathcal F_*$, every $n\geq 1$ and every $z\in D^*$ it holds that
$$\left\vert m_{\tilde f^n, \phi}(z) -\mathrm{ang}_{I^n} (0,z)\right\vert\leq 1.$$
Moreover, the function $z\mapsto \mathrm{ang}_{I} (0,z)$ is bounded because $f$ is a $C^1$ diffeomorphism. It implies that $m_{\tilde f, \phi}$ is bounded and so, $\mu_{\omega_{\omega}} $ has a rotation number (according to Definition  \ref{de:rotationlinkingmeasure}) and we have
$$\mathrm{rot}_{\tilde f}(\mu)=\int_{D^*} m_{\tilde f, \phi}\, d\mu_{\omega}=\int_{D^*} \mathrm{ang}_{I} (0,z)\, d\mu_{\omega}.$$

Note now that for every $n\geq 1$ and every $(z,z')\in W$ it holds that
$$\left\vert {\Lambda}_{\tilde f^n, \mathcal F_*, \phi} (z,z')-\mathrm{ang}_I^n(z,z')\right\vert\leq 1+1=2.$$
To get this inequality, one must consider a frame at $z$ moving with time: at time $n$ it is the image of the original frame at $z$ by the rotation $R_{\mathrm{ang}_{I^n} (0,z)}$. Indeed,  we have the two following properties:

\begin{itemize}
\item  $\left\vert m_{\tilde f^n, \phi}(z) -\mathrm{ang}_{I^n} (0,z)\right\vert\leq 1$;
\item the difference between the variation of 
angle of the vector $f_s(z')-f_s(z)$, $s\in[0,n]$, in the moving frame and  ${\lambda}_{ f^n, \mathcal F_*} (z,z')$ is smaller than $1$.
\end{itemize}

One deduces that $\mu_{\omega}$ has a self-linking number (according to Definition \ref{de:rotationlinkingmeasure}) and we have

$$\mathrm{link}_{\tilde f}(\mu_{\omega})= \int_{W} {\Lambda}_{\tilde f, \mathcal F, \phi}\, d\mu_{\omega}\times d\mu_{\omega}= \int_{W}\mathrm{ang}_I^n(z,z')\, d\mu_{\omega}\times d\mu_{\omega} =\widetilde{\mathrm{Cal}}(I).$$

Let us conclude with a proposition, that will useful to prove Theorem \ref{th:principal}, and that summarize what has been done in this section.

\begin{proposition}Suppose that $f\in\mathrm{Diff}^1_{\omega}(\D)$ fixes $0$. Let $I$ be an identity isotopy of $f$ and $\tilde f$ be the lift of $f_{\vert \D^*}$ naturally defined by $[I]$. If  $\mathcal F\in\mathfrak{F}$ is a radial foliation and $\phi$ is a leaf of $\mathcal F$ such that $m_{\tilde f, \phi}$ is $\mu_{\omega}$-integrable and
${\lambda}_{f, \mathcal F}$ is $\mu_{\omega}\times\mu_{\omega}$-integrable, then it holds that 
$$\ \widetilde{\mathrm{Cal}}(I)= \int_{W} {\Lambda}_{\tilde f, \mathcal F, \phi}\, d\mu_{\omega}\times d\mu_{\omega}.$$ 
\end{proposition}

\section{Construction of a good radial foliation for an irrational pseudo-rotation.} \label{se:goodfoliation}

The three first sections of this chapter come from \cite{L2}. All proofs can be found there. The fourth one, concerning properties of projected foliations is mainly new. The last proposition of the fourth section and the isotopy defined in the last section already appeared in \cite{L1} but in a slightly different context (twist maps instead of untwisted maps).

\subsection{Generating functions}\label{sse:generating} Let us denote $\pi_1:(x,y)\mapsto x$ and $\pi_2:(x,y)\mapsto y$ the two projections  defined on the Euclidean plane $\R^2$. An orientation preserving homeomorphism $f$ of $\R^2$  will be called {\it untwisted} if the map
$$(x,y)\mapsto(\pi_1(f(x,y)), y)$$ is a homeomorphism, which means that there exist two continuous functions $g, g'$ on $\R^2$ such that
$$f(x,y)=(X,Y)\Leftrightarrow\begin{cases} x=g(X,y),\\ Y=g'(X,y).\end{cases}$$ In this case, the maps $X\mapsto g(X,y)$ and $y\mapsto g'(X,y)$ are orientation preserving homeomorphisms of $\R$. If moreover, $f$ is area preserving, the continuous form $x dy+Y dX$ is exact: there exists a  $C^1$ function $h: \R^2\to \R$ such that 
$$\displaystyle g={\partial h\over \partial y}, \enskip\displaystyle g'={\partial h\over \partial X}.$$The function $h$, defined up to an additive constant, is a  {\it generating function} of $f$. 

\medskip
We can precise the definition by saying that  $f$  is a {\it $K$ Lipschitz untwisted homeomorphism}, where $K\geq 1$, if

\smallskip
\noindent {\bf i)}\enskip \enskip $f$ is untwisted;

\smallskip
\noindent {\bf ii)}\enskip \enskip $f$ is $K$ bi-Lipschitz;

\smallskip
\noindent {\bf iii)}\enskip \enskip the maps $X\mapsto g(X,y)$ and $y\mapsto g'(X,y)$ are $K$ bi-Lipschitz;

\smallskip
\noindent {\bf iv)} \enskip \enskip the maps $y\mapsto g(X,y)$ and $X\mapsto g'(X,y)$ are $K$ Lipschitz.

\medskip
If $f$ is a diffeomorphism of $\R^2$, denote $\mathrm{Jac}(f)(z)$ the Jacobian matrix at a point $z$. One proves easily, that for every $K>1$, there exists a neighborhood $\mathcal U$ of the identity matrix in the space of square matrices of order $2$, such that every $C^1$ diffeomorphism satisfying  $\mathrm{Jac}(f)(z)\in\mathcal U$, for every $z\in\R^2$, is a  $K$ Lipschitz untwisted homeomorphism.

Suppose that $f$ is a $C^1$ orientation preserving diffeomorphism of $\D$ that fixes $0$ and coincides with a rotation $R_{\alpha}$ on $\S$. Fix $\beta>\alpha$. We can extend our map to a homeomorphism of the whole plane (also denoted $f$) such that:$$
 f(z)= \begin{cases} R_{\alpha+r-1}(z)&\text{ if \enskip $1\leq \vert z\vert \leq1+\beta-\alpha,$}\\
R_{\beta}(z)&\text{ if \enskip $\vert z\vert\geq 1+\beta-\alpha$.}\end{cases}
$$

Using the fact that the group of orientation preserving $C^1$ diffeomorphisms of $\D$ (and the group of symplectic diffeomorphisms of $\D$) that fix $0$ and every point of $\S$, when furnished with the $C^1$ topology, is path connected, one can prove the following:

\begin{proposition} \label{pr:decomposition} For every $K>1$, one can find a decomposition $f=f_m\circ\dots \circ f_1$, where each $f_i$ is a $K$ Lipschitz untwisted homeomorphism that fixes $0$ and induces a rotation on every circle of origin ${0}$ and radius $r\geq 1$. Moreover, if $f$ is area preserving, one can suppose that each $f_i$ preserves the area. \end{proposition}

We fix $K>1$ and a decomposition $f=f_m\circ\dots \circ f_1$ given by Proposition \ref{pr:decomposition}. We define two families $(g_i)_{1\leq i\leq m}$, $(g'_i)_{1\leq i\leq m}$ of continuous maps as follows
$$f_i(x,y)=(X,Y)\Leftrightarrow\begin{cases} x=g_i(X,y),\\ Y=g'_i(X,y),\end{cases}$$
For every $i\in\{1,\dots, m\}$ one can find an identity isotopy $I_i=(f_{i,s})_{s\in[0,1]}$ where each $f_{i,s}$ is an untwisted map such that
$$f_{i,s}(x,y)=(X,Y)\Leftrightarrow\begin{cases} x=(1-s) X+ sg_i(X,y),\\ Y=(1-s)y+ sg'_i(X,y),\end{cases}$$ and a natural isotopy $I=I_m\circ \dots \circ I_1$ of $f$. Note that if each $f_i$ is area preserving, one can find a family $(h_i)_{1\leq i\leq m}$ of $C^1$ maps, such that$$\displaystyle g_i={\partial h_i\over \partial y}, \enskip\displaystyle g'_i={\partial h_i\over \partial X}.$$ In that case every map $f_{i,s}$, $1\leq i\leq m$, $s\in[0,1]$, is area preserving and it holds that $$(X,y)\mapsto (1-s)Xy+s h_i(X,y)$$ is a generating function of $f_{i,s}$.

\subsection{The vector field associated to a decomposition} \label{sse:vectorfield}We consider in this section a  $C^1$ orientation preserving diffeomorphism $f$ of $\D$ that fixes $0$ and coincides with a rotation $R_{\alpha}$ on $\S$. We suppose that it is extended and then decomposed into untwisted maps as in Section \ref{sse:generating}. We keep the same notations. We extend the families 
$$(f_i)_{1\leq i\leq m}, \enskip (g_i)_{1\leq i\leq m},  \enskip (g'_i)_{1\leq i\leq m},  $$ to $m$ periodic families $$(f_i)_{i\in\Z},  \enskip (g_i)_{i\in\Z},  \enskip (g'_i)_{i\in\Z},$$
and the family $ (h_i)_{1\leq i\leq m} $ to a  $m$ periodic family  $ (h_i)_{i\in\Z}$ in case the $f_i$ are area preserving.

We fix an integer $b\geq 1$ and consider the finite dimensional vector space
$$	E_b= \left\lbrace {\bf z}=(z_{i})_{i\in {\Z}}\in (\R^2)^{\Z}\enskip 
\enskip\vert \enskip z_{i+mb}=z_{i},\enskip 
\mathrm{ for\enskip all}\enskip i\in {\Z}\right\rbrace,
$$
furnished with the scalar product
$$\left\langle
(z_{i})_{i\in {\Z}}, (z'_{i})_{i\in {\Z}}\right\rangle=\sum_{0< i\leq mb} x_ix'_i+y_iy'_i,$$ where $z_{i}=(x_i,y_i)$ and $z'_{i}=(x'_i,y'_i)$.
\medskip
We define on $E_b$ a vector field $\zeta=(\zeta_i)_{i\in\Z}$ by writing
$$
{\bf \zeta}
_{i}({\bf z})= ({\bf \xi}
_{i}({\bf z}), {\bf \eta}
_{i}({\bf z})) = \left(y_i-g'_{i-1}(x_{i},y_{i-1}), \enskip x_i-g_i(x_{i+1},y_i)\right).$$ Observe that $\zeta$ is invariant by the ($b$ periodic) shift
$$\eqalign{\varphi : E_b&\to E_b\enskip ,\cr
(z_{i})_{i\in {\Z}}&\mapsto (z_{i+m})_{i\in 
{\Z}} .\cr}$$

Let us state some facts about  $\zeta$. 
\begin{lemma} \label{le:lipschitz}The vector field is $A$ Lipschitz, where $A=\sqrt{6K^2+3}$.\end{lemma}

One deduces that the associated differential system 
$$\begin{cases} \dot x_i= y_i-g'_{i-1}(x_{i},y_{i-1}),\\ \dot y_i= x_i-g_i(x_{i+1},y_i),\end{cases}$$
defines a flow on $E$. We will denote by ${\bf z}^t$ the image at time $t$ of a point 
${\bf z}\in E$ by this flow.  As an application of Gronwall's Lemma, one gets:

\begin{lemma} \label{le:gronwall}For every $({\bf z}, {\bf z}')\in E_b$ and every $t\in\R$, one has 
$$e^{-A\vert t\vert}\Vert{\bf z}-{\bf z}'\Vert \leq\Vert {\bf z}^t-{\bf z}'{}^t\Vert \leq  e^{A\vert t\vert}\Vert{\bf z}-{\bf z}'\Vert $$and
$$e^{-A\vert t\vert}\Vert \zeta({\bf z})\Vert \leq\Vert \zeta({\bf z}^t)\Vert \leq  e^{A\vert t\vert}\Vert \zeta({\bf z})\Vert 
.$$\end{lemma}

In the case where the $f_i$ are area preserving, observe that $\zeta$ is the gradient
vector field of the function
$${\bf h}:{\bf z}\mapsto\sum_{0<i\leq mb} x_iy_i-h_{i-1}(x_i,y_{i-1})$$ and that ${\bf h}$ is invariant by $\varphi$. One can define the {\it energy} of an orbit $({\bf z}^t)_{t\in\R}$ to be
$$ \int_{-\infty}^{+\infty} \Vert \zeta({\bf z}^t)\Vert^2 \, dt= \lim_{t\to+\infty}  {\bf h}({\bf z^t})- \lim_{t\to-\infty}  {\bf h}({\bf z^t}).$$
As a consequence of Lemma  \ref{le:lipschitz} it holds that

\begin{lemma} \label{le:gradient} For every ${\bf z}\in E_b$, one has
$$\Vert \zeta({\bf z})\Vert^2\leq  A\int_{-\infty}^{+\infty} \Vert \zeta({\bf z}^t)\Vert^2 \, dt = A\left(\lim_{t\to+\infty}  {\bf h}({\bf z^t})- \lim_{t\to-\infty}  {\bf h}({\bf z^t})\right).$$
\end{lemma}

For every $i\in\Z$, define the maps $Q_i,\, P_i, \,Q'_i: E_b\to \R^2$, where 
$$ Q_i({\bf z})= (g_{i}(x_{i+1},y_{i}), y_i), \enskip P_i({\bf z})= (x_{i}, y_i), \enskip Q'_i({\bf z})=(x_{i},g'_{i-1}(x_{i},y_{i-1})).$$

Let us state the main properties of these maps:

\begin{itemize}
\item $f_i\circ Q_i=Q'_{i+1},$ 

\item  $\zeta_i=J\circ (Q'_i-Q_i)$ where
$J(x,y)=(-y,x)$,

\item  ${\bf z}\in E$ is a singularity of $\zeta$ if and only if $Q_i({\bf z})=Q'_{i}({\bf z})$ for every $i\in\Z$,

\item  if ${\bf z}\in E$ is a singularity of $\zeta$ then $Q_i({\bf z})=P_{i}({\bf z})=Q'_{i}({\bf z})$ for every $i\in\Z$, 

\item $Q_{1}$ induces a bijection between the singular set of
$\zeta$ and the fixed point set of $f^b$, 

\item the  sequence ${\bf 0}=(0)_{i\in\Z}$ is a singular point of $\zeta$ that is sent onto $0$ by each $Q_i$, $P_i$  or $Q'_i$,

\item $\zeta$ is $C^1$ in a neighborhood of ${\bf 0}$. 

\end{itemize}

\subsection{The case of an irrational pseudo-rotation} 
In this section we keep the notation of Section \ref{sse:vectorfield} but we suppose than $f$ is an irrational pseudo-rotation. We suppose moreover than $\beta\not\in\Q$ and that $(\alpha,\beta)\cap\Z=\emptyset$. The extension $f$  is a piecewise $C^1$ area preserving transformation that satisfies the following properties:

\smallskip
\noindent-\enskip\enskip $0$ is the unique fixed point of $f$;

\smallskip
\noindent-\enskip\enskip there is no periodic point of period $b$ if $(b \alpha, b \beta)\cap\Z=\emptyset$;

\smallskip
\noindent-\enskip\enskip if $(b \alpha, b \beta)\cap\Z\not=\emptyset$, the set of periodic points of period $b$ can be written $\bigcup_{\alpha<a/b<\beta} S_{a/b}$, where $S_{a/b}$ is the circle of center $0$ and radius $1+a/b- \alpha$.

\bigskip
In that case, we have the additional following properties for the vector $\zeta$ defined on $E_b$:

\begin{itemize}

\item the singular set consists of the constant sequence ${\bf 0}$ and of finitely many smooth closed curves $(\Sigma_a)_{a\in (b\alpha, b\beta)\cap {\Z}}$;

\item the curve $\Sigma_a$ is sent homeomorphically onto $S_{a/b}$ by each $Q_i$, $P_i$ or $Q'_i$;

\item $\xi$ is $C^{\infty}$ in a neighborhood of $\Sigma_a$. 

\end{itemize}

 We fix an integer $b\geq 2$ such that $(b\alpha, b\beta)\cap {\Z}\not=\emptyset$. If ${\bf z}$ and ${\bf z'}$ are two singular points of $\zeta$,  the quantity ${\bf h}({\bf z})-{\bf h}({\bf 0})$ is the difference of {\it action} between the two corresponding fixed points of $f^b$.  A computation gives us:

\begin{lemma} \label{le:computationaction}
For every ${\bf z}\in\Sigma_a$, one has
$${\bf h}({\bf z})-{\bf h}({\bf 0})=\pi(a-b \alpha) \left( 1 + (a/b -\alpha) +{(a/b-\alpha)^2\over 3}\right).$$\end{lemma}

We will denote 
$$C(a,b)=\pi(a-b \alpha) \left( 1 + (a/b -\alpha) +{(a/b-\alpha)^2\over 3}\right).$$

Now let us state the fundamental result of \cite{L2}:

\begin{proposition}\label{pr:invariantdisk}
The curve $\Sigma_a$ bounds a topological disk $\Delta_{a}\subset E_b$ that satisfies the following:

\smallskip 
\noindent{\bf i)}\enskip $\Delta_{a}$ contains the constant sequence ${\bf 0}$;

\smallskip 
\noindent{\bf ii)}\enskip $\Delta_{a}$ is invariant by $\varphi$;

\smallskip 
\noindent{\bf iii)}\enskip each projection ${\bf z}\mapsto (x_i,y_{i-1})$, $i\in\Z$, is one to one on $\Delta_{a}$;

\smallskip 
\noindent{\bf iv)}\enskip each projection ${\bf z}\mapsto (x_i,y_i)$, $i\in\Z$,  is one to one on $\Delta_{a}$;

\smallskip 
\noindent{\bf v)}\enskip $\Delta_{a}$ is invariant by the flow of $\zeta$;

\smallskip 
\noindent{\bf vi)}\enskip for every ${\bf z}\in \Delta_a^*=\Delta_{a}\setminus (\{{\bf 0}\}\cup\Sigma_a)$, one has $\lim_{t\to-\infty} {\bf z}^t={\bf 0}$ and $\lim_{t\to+\infty} d({\bf z}^t, \Sigma_a)=0$. 

\end{proposition}

Let us explain more precisely what is proved in \cite{L2}.
 In what follows, the function $\mathrm{sign}$ assigns $+1$ to a positive number and $-1$ to a negative number.

Let us consider the set
$$V=\{{\bf z}\in E_b\enskip\vert\,
x_i\not=0\enskip \mathrm{and}\enskip  y_i\not=0\enskip{\rm for\enskip all}\enskip i\in\Z\}$$ and the function $L$ on $V$ defined by the formula
$$\eqalign{L({\bf z})&={1\over 4}\sum_{0<i\leq mb} \mathrm{sign}( x_i) \left( \mathrm{sign}(y_i)-\mathrm{sign} (y_{i-1})\right),\cr
&= {1\over 4}\sum_{0<i\leq mb}  \mathrm{sign}( y_i) \left( \mathrm{sign}(x_i)-\mathrm{sign} (x_{i+1})\right).\cr}$$
It extends
continuously to the open set 
$$V' =\{{\bf z}\in E_b\enskip\vert\enskip x_i=0\Rightarrow y_{i-1}y_{i}>0, \enskip  y_i=0\Rightarrow x_{i}x_{i+1}>0\}.$$
It is integer valued and takes its values in $\{-[mb/2],\dots, [mb/2]\}$, where the notation $[x]$ denotes the integer part of a real number $x$.  The assertion {\bf iii)} and {\bf iv)} are immediate consequences of the following fact:

\begin{proposition}\label{pr:invariantdisklinkingnumber} If ${\bf z}$ and ${\bf z'}$ are two different points of $\Delta _a$, then ${\bf z}-{\bf z'}\in V'$ and $L({\bf z}-{\bf z'})=a$. \end{proposition}
\medskip

The fundamental result that permits to construct $\Delta _a$ is the following:

\begin{proposition}\label{pr:linking} If
${\bf z}$, 
${\bf z}$ are two distinct points of $E_b$ satisfying ${\bf z}'-{\bf z}\not\in V'$, then there exists $\varepsilon >0$ such
that for every $t\in(0,\varepsilon]$, it holds that:
$${\bf z'}^{-t}-{\bf z}^{-t}\in V', \enskip {\bf z'}^{t}-{\bf z}^{t}\in V', \enskip L({\bf z'}^{-t} -{\bf z}^{-t})<L({\bf z'}^{t} -{\bf z}^{t}).$$\end{proposition}
\bigskip

This result admits an infinitesimal version (see \cite{L1}):

\begin{proposition}\label{pr:infinitesimallinking}  If
${\bf z}\in E_b$ is non singular and satisfies $\zeta({\bf z})\not\in W$, then there exists $\varepsilon >0$ such
that for every $t\in(0,\varepsilon]$, it holds that:
$$\zeta({\bf z'}^{-t})\in V', \enskip\zeta({\bf z'}^{t})\in V', \enskip L(\zeta({\bf z'}^{-t}))<L(\zeta({\bf z'}^{t}))$$and
$${\bf z'}^{-t}-{\bf z}\in V', \enskip {\bf z'}^{t}-{\bf z}\in V', \enskip L({\bf z'}^{-t} -{\bf z})<L({\bf z'}^{t} -{\bf z}).$$\end{proposition}

As an immediate corollary of the second assertion, one gets the following result (not explicitely stated in \cite{L2})

\begin{corollary}\label{co:linkinginfinitesimal} If ${\bf z}\in\Delta _a^*$, then $\zeta({\bf z})\in V'$ and $L(\zeta({\bf z}))=a$. \end{corollary}

The assertion {\bf iii)} tells us that the maps $Q_i\vert_{\Delta_{a}}$ and $Q'_i\vert_{\Delta_{a}}$, $i\in\Z$, induce homeomorphisms from $\Delta_{a}$ to $D_{a/b}$ denoted respectively $q_i$ and $q'_i$ \footnote{ We do not refer to $a$ to lighten the notations.}. The assertion {\bf iv)} tells us that the maps $P_i\vert_{\Delta_{a}}$, $i\in\Z$, induce homeomorphisms from $\Delta_{a}$ to $D_{a/b}$ denoted $p_i$. Note that $p_i\circ q_i^{-1}$ is a homeomorphism which let invariant every horizontal segment of $D_{a/b}$ and induces on this segment an increasing homeomorphism and similarly $p_i\circ q'{}_i^{-1}$ is a homeomorphism which let invariant every vertical segment of $D_{a/b}$ and induces on this segment an increasing homeomorphism. We denote ${\mathcal F}$ the radial foliation defined on $\Delta_a^*$ whose leaves are the gradient lines of $\zeta$ included in $\Delta_a^*$.

For every $s\in[0,1]$, and every $i\in\Z$ we define 
$$q^s_i= (1-s)q_i +s p_i,  \enskip q'{}^{s}_i= (1-s) p_i+ s q'_i.$$ The maps  $q^s_i$ and $q'{}^{s}_i$ are homeomorphisms from $\Delta_{a}$ to $D_{a/b}$ (all coinciding on $\Sigma_a$) and send $\mathcal F$ onto a radial foliation of $D_{a/b}$.

\subsection{Projected radial foliations} Recall that $d$ denotes the winding distance between two radial foliations. Let us begin with the following result:
\begin{lemma} \label{le:distancefoliations}
 For every $s_1, s_2$ in $[0,1]$ and every $i\in\Z$, we have
$$d(q_i^{s_1}(\mathcal F), q_i^{s_2}(\mathcal F))\leq 2, \enskip d(q'{}_i^{s_1}(\mathcal F), q'{}_i^{s_2}(\mathcal F))\leq 2.$$\end{lemma}

 \begin{proof}We fix $i\in\Z$ and will prove the first inequality, the proof of the second one being similar. The leaves of $q_i^s(\mathcal F)$ are the paths $t\mapsto q_i^s({\bf z}^t)$, ${\bf z}\in \Delta_a^*$. Note that the map $t\mapsto \pi_2(q_i^s({\bf z}^t))$ is $C^1$ and that
$$ {d\over dt}  \pi_2\circ q^s_i({\bf z}^t)_{\vert t=0}=  \eta_i({\bf z})= x_i-g_i(x_{i+1}, y_i).$$ 
Say that $z\in D_{a/b}^*$ is {\it horizontal } if it can be written $z=q_i({\bf z})$, where $x_i-g_i(x_{i+1}, y_i)=0$. Denote $H_i$ the set of horizontal points\footnote{ if $\eta_i({\bf z})=0$, then $\xi_i({\bf z})\not=0$, so every leaf of $p_i(\mathcal F)$ is a $C^1$ embedded line and $H_i$ is nothing but the set of points where the foliation $p_i(\mathcal F)$ is horizontal.}. For every point $z\in D_{a/b}^*$, define $J(z)$ to be equal to $\{z\}$ if $z$ is horizontal and to the largest horizontal interval contained in $D_{a/b}^*\setminus H_i$ and containing $z$, if $z$ is not horizontal. The sign of $\eta_i(q_i^{-1}(z'))$ does not depend on the choice of $z'\in J(z)$. We orient $J(z)$ with $x$ increasing if this sign is positive and with $x$ decreasing if this sign is negative. Note that the path $s\mapsto u_i^s(z)= q_i^s\circ q_i^{-1}(z)$ is constant if $z$ is horizontal and is an oriented segment of $J(z)$ inheriting the same orientation if $z$ is not horizontal. The fact that $\displaystyle {d\over dt}  \pi_2\circ q^{s}_i({\bf z}^t)_{\vert t=0}=  \eta_i({\bf z})$, for every $s\in[0,1]$ and every ${\bf z}\in \Delta_a^*$, implies that if $z$ is not horizontal, the oriented interval $J(z)$ is transverse to the foliations $q_i^{s}(\mathcal F)$, $s\in [0,1]$, crossing locally every leaf from the right to the left.

Denote $\tilde D_{a/b}$ the universal covering space of $D_{a/b}^*$ and for every $\tilde z\in \tilde D_{a/b}$ that lifts  $z\in D_{a/b}^*$, denote $\tilde J(\tilde z)$ the lift of $J(z)$ containing $\tilde z$, with the induced orientation in case $z\not\in H_i$. Write $\tilde H_i$ for the lift of $H_i$, write $\widetilde{q_i^{s}(\mathcal F)}$ for the lift of $q_i^{s}(\mathcal F)$ and denote $(\tilde u_i^s)_{s\in[0,1]}$  the identity isotopy that lifts $(u_i^s)_{s\in[0,1]}$. Suppose that $\tilde z\in \tilde D_{a/b}$ is not in $\tilde H_i$. For every $s_0\in[0,1]$,  the oriented interval $\tilde I(\tilde z)$ is transverse to the foliation $\widetilde{q_i^{s_0}(\mathcal F)}$, crossing locally every leaf from the right to the left. Moreover if $0\leq s_1< s_2$, then $\tilde u_i^{s_2}(\tilde z)$ follows $\tilde u_i^{s_1}(\tilde z)$ on $\tilde J(\tilde z)$. Consequently it holds that:
$$ 0\leq s_0\leq 1 \enskip \mathrm {and} \enskip  0\leq s_1< s_2\leq 1 \Rightarrow \theta_{\tilde u_i^{s_1}(\tilde z), \tilde u_i^{s_2}(\tilde z)}(q_i^{s_0}(\mathcal F) )=\dot 1.$$

Now, let us fix two different points $\tilde z_0$ and $\tilde z_1$ in $\tilde D_{a/b}$ and $s\in[0,1]$.   Observe that the three sets 
$$\eqalign{\{(\tilde z'_0, \tilde z'_1) \in \tilde J(\tilde z_0) \times  \tilde J(\tilde z_1) \,&\vert\, \,\tilde z'_0\not =\tilde z'_1,\, \theta_{\tilde z'_0, \tilde z'_1}( q_i^{s}(\mathcal F))=\dot 1\},\cr
\{(\tilde z'_0, \tilde z'_1) \in \tilde J(\tilde z_0) \times  \tilde J(\tilde z_1) \,&\vert\, \,\tilde z'_0\not =\tilde z'_1,\, \theta_{\tilde z'_0, \tilde z'_1}( q_i^{s}(\mathcal F))=\dot 3\},\cr
\{(\tilde z'_0, \tilde z'_1) \in \tilde J(\tilde z_0) \times  \tilde J(\tilde z_1) \,&\vert\, \,\tilde z'_0\not =\tilde z'_1,\, \theta_{\tilde z'_0, \tilde z'_1}( q_i^{s}(\mathcal F))\in\{\dot0, \dot 2\}\}\cr}$$ are connected, whether the points belong to $\tilde H_i$ or not. Indeed, the paths $\tilde J(\tilde z_0)$ and $\tilde J(\tilde z_1)$ in $\tilde D_{a/b}$ draw intervals of leaves of the foliation $\widetilde{q_i^{s}(\mathcal F)}$ whose intersection is an interval of leaves if not empty. In particular the set of couples $(\tilde z'_0, \tilde z'_1) \in \tilde J(\tilde z_0) \times  \tilde J(\tilde z_1)$ such that $\tilde z'_0$ and  $\tilde z'_1$ are distinct and belong to the same leaf of $\widetilde {q_i^{s}(\mathcal F)}$ is connected: it is an interval (possibly empty) if $\tilde J(\tilde z_0)\cap \tilde J(\tilde z_1)=\emptyset$, it is empty if  $\tilde J(\tilde z_0)\cap \tilde J(\tilde z_1)\not=\emptyset$ (because the intersection is a ``horizontal'' path in $\tilde D_{a/b}$ ). The map $(\tilde z'_0, \tilde z'_1)\mapsto  \theta_{\tilde z'_0, \tilde z'_1}( q_i^{s}(\mathcal F))$ being continuous on $\tilde W$, takes at most three values on  $\tilde J(\tilde z_0) \times  \tilde J(\tilde z_1)$  (either $\dot 0$ or $\dot 2$ is missing).   

In particular, if $0\leq s_1< s_2\leq 1$, the map 
$$s\in[s_1,s_2]\mapsto \theta_{\tilde u_i^s\circ (\tilde u_i^{s_1})^{-1}(\tilde z_0), \,\tilde u_i^s\circ (\tilde u_i^{s_1})^{-1}(\tilde z_1)} (q_i^{s_2}(\mathcal F))=\theta_{\tilde z_0,\tilde z_1} (  u_i^{s_1}\circ (u_i^{s})^{-1}\circ q_i^{s_2}(\mathcal F))\in \Z/4\Z$$ takes at most three values and is lifted to a map 
$$s\in[s_1,s_2]\mapsto \hat \theta_{\tilde z_0,\tilde z_1} ( u_i^{s_1}\circ (u_i^{s})^{-1}\circ q_i^{s_2}(\mathcal F))\in \Z$$  taking at most three values, which implies that  $$\vert\hat \tau (\tilde z_0, \tilde z_1, q_i^{s_2}(\mathcal F), q_i^{s_1}(\mathcal F))\vert \leq 2.$$
\end{proof}

The homeomorphism $\tilde u_i^s$ sends the foliation  $\widetilde{q_i(\mathcal F)}$ onto the foliation  $\widetilde{q_i^{s}(\mathcal F)}$. The leaves of  $\widetilde{q_i(\mathcal F)}$ are pushed on the left by the isotopy $(\tilde u_i^s)_{s\in[0,1]}$.  More precisely, for every $0\leq s_1<s_2\leq 1$ and every leaf $\tilde \phi$ of $\widetilde{q_i(\mathcal F)}$, it holds that:
\begin{itemize}
\item $L( \tilde u_i^{s_2}(\tilde \phi))\subset L( \tilde u_i^{s_1}(\tilde \phi))$,
\item if $\tilde z\in\tilde\phi$ belongs to $\tilde H_i$, then $\tilde z \in\tilde u_i^{s_1}(\tilde \phi) \cap \tilde u_i^{s_2}(\tilde \phi)$,
\item if $\tilde z\in\tilde\phi$ does not belong to $\tilde H_i$, then $\tilde u_i^{s_2}(\tilde z)\in L(\tilde u_i^{s_1}(\tilde \phi))$.
\end{itemize} 

\bigskip Similarly,   for every ${\bf z} \in \Delta_{a}^*$, the map $t\mapsto \pi_1(q'{}_i^s({\bf z}^t))$ is $C^1$ and we have
$$ {d\over dt}  \pi_1\circ q'{}^s_i({\bf z}^t)_{\vert t=0}=  \xi_i({\bf z})= y_i-g_{i-1}(x_{i}, y_{i-1}).$$ 
Say that $z\in D_{a/b}^*$ is {\it vertical} if it can be written $z=q'_i({\bf z})$, where $y_i-g_{i-1}(x_{i}, y_{i-1})=0$. Denote $V_i$ the set of vertical points and $\tilde V_i$ its lift in $\tilde D_{a/b}$.
Define $u'{}_i^s= q'{}_i^s\circ p_i^{-1}$ and lift the isotopy $(u'{}_i^s)_{s\in[0,1]}$ to an identity isotopy $(\tilde u'{}_i^s)_{s\in[0,1]}$. Write  $\widetilde{q'{}_i^{s}(\mathcal F)}$ for the lift of $q'{}_i^{s}(\mathcal F)$.  Then, for every $0\leq s_1<s_2\leq 1$ and every every leaf $\tilde \phi$ of $\widetilde{q'{}_i(\mathcal F)}$, it holds that:

\begin{itemize}
\item $L( \tilde u'{}_i^{s_2}(\tilde \phi))\subset L( \tilde u'{}_i^{s_1}(\tilde \phi))$,
\item if $\tilde z\in\tilde\phi$ belongs to $\tilde V_i$, then $\tilde z \in\tilde u'{}_i^{s_1}(\tilde \phi) \cap \tilde u'{}_i^{s_2}(\tilde \phi)$,
\item if $\tilde z\in\tilde\phi$ does not belong to $\tilde V_i$, then $\tilde u'{}_i^{s'}(\tilde z)\in L(\tilde u'{}_i^{s}(\tilde \phi))$.
\end{itemize}

To conclude, consider the family $(v_i^s)_{s\in[0,2]}$ where
$$v_i(s)= \begin{cases} u_i^s &\text{ if $0\leq s\leq  1$,}\\  u'{}_i^{s-1} \circ u_i^1= q'{}_i^{s-1}\circ q_i^{-1} &\text{ if  $1\leq s\leq 2$,}\end{cases}$$ 
 is an isotopy from $\mathrm{Id}$ to $q'_i\circ q^{-1}_{i}$.   The isotopy $\left(v_i^s{}_{\vert  D_{a/b}^*}\right)_{s\in[0,2]}$ can be lifted to an identity isotopy $\left(\tilde v_i^s\right)_{s\in[0,2]}$ such that 
  $$\tilde v_i(s)= \begin{cases}\tilde u_i^s &\text{ if $0\leq s\leq  1$,}\\  \tilde u'{}_i^{s-1} \circ \tilde u_i^1 &\text{ if  $1\leq s\leq 2$.}\end{cases}$$
  The map $\tilde v_i^s$ sends the foliation  $\widetilde{q_i(\mathcal F)}$ onto the foliation $\widetilde{q^s_i(\mathcal F)}$ if $0\leq s\leq 1$ and onto the foliation $\widetilde {q'{}_i^{s-1}(\mathcal F)}$ if $1\leq s\leq 2$. Moreover: 
  \begin{itemize}
\item for every $0\leq s_1<1<s_2\leq 2$ and every leaf $\tilde \phi$ of $\widetilde{\mathcal F}_i^0$, it holds that $\overline{L( \tilde v_i^{s_2}(\tilde \phi))}\subset L( \tilde v_i^{s_1}(\tilde \phi)$.
\end{itemize} 
Indeed, the sets $H_i$ and $V_i$ are disjoint.

\subsection{Construction of a good isotopy}

One gets a $m$-periodic family of homeomorphisms $(\widehat f_i)_{i\in\Z}$ of $D_{a/b}$ by writing:
$$\widehat f_i=q_{i+1}\circ q_i^{-1}.$$
Its periodicity comes from the equalities
$$\widehat f_{i+m}=q_{i+m+1}\circ q_{i+m}^{-1}= q_{i+1}\circ \varphi_{\vert\Delta_a}\circ  (q_{i}\circ \varphi_{\vert\Delta_a})^{-1}$$

Moreover  $\widehat f =\widehat f_m\circ\dots\circ \widehat f_1$ has order $q$ because
$$\widehat f =q_{m+1}\circ q_{1}^{-1}=q_{1}\circ (\varphi_{\vert\Delta_a})\circ q_{1}^{-1}.$$
or equivalently because
$$\widehat f^b= \widehat f_{mb}\circ\dots\circ \widehat f_1=q_{mb+1}\circ q_{1}^{-1}=\mathrm{Id}_{D_{a/b}}.$$

Note also that $\widehat f_i$ fixes $0$ and coincides with $f_i$ on $S_{a/b}$ because
$$f_i\vert_{D_{a/b}}=q'_{i+1}\circ q_i^{-1}.$$

Let us define an isotopy $\check I=(\check f_s)_{s\in[0,mb]}$ from $\mathrm{Id}$ to $f^b$ in the following way:

\medskip

If  $s\in [2k,2k+1]$, $0\leq k<mb$, then 
$$\begin{aligned}\check f_s&= f_{mq}\circ \dots \circ f_{mb-k+1} \circ q_{mb-k+1}^{s-2k}\circ q_1^{-1}\\
&=  f_{mq}\circ \dots\circ f_{mb-k+1}  \circ q_{mb-k+1}^{s-2k}\circ q^{-1}_{mq-k}\circ \widehat f_{mb-k-1}\circ \dots\circ \widehat f_1.\end{aligned}$$

If  $s\in [2k+1,2k+2]$, $0\leq k<mq$, then 
$$\begin{aligned}\check f_s&= f_{mb}\circ \dots \circ f_{mb-k+1} \circ q'{}_{mb-k+1}^{s-2k-1}\circ q_1^{-1}\\
&=  f_{mq}\circ \dots\circ f_{mb-k+1}  \circ q'{}_{mb-k+1}^{s-2k-1}\circ q^{-1}_{mb-k}\circ \widehat f_{mb-k-1}\circ \dots\circ \widehat f_1.\end{aligned}$$

Note that
$$\begin{aligned}\check f_{2k}&= f_{mb}\circ \dots\circ  f_{mb-k+1}\circ \widehat f_{mb-k}\circ \dots\circ \widehat f_1,\\
 \check f_{2k+1}&= f_{mb}\circ \dots\circ  f_{mb-k+1}\circ p_{mb-k+1}\circ q_{mb-k}^{-1}\circ\widehat  f_{mb-k-1}\circ \dots\circ \widehat f_1.\end{aligned}
$$

In particular, we have 

$$\check f_s(q_1({\mathcal F}))= \begin{cases}  f_{mb}\circ \dots \circ f_{mb-k+1} (q_i^{s-2k}({\mathcal F})) &\text{ if $s\in [2k,2k+1]$,}\\
f_{mb}\circ \dots \circ f_{mb-k+1} (q'{}_i^{s-2k-1}({\mathcal F})) &\text{ if $s\in [2k+1,2k+2]$.}\end{cases}$$

\begin{proposition} \label{pr:distancefoliationsglobal} For every $s_1$, $s_2$ in $[0,2mb]$ we have
$$d\left(\check f_{s_1}(q_1({\mathcal F})), \check f_{s_2}(q_1({\mathcal F}))\right)\leq 2\vert s_2-s_1\vert +4.$$ Moreover, if $s_1$ and $s_2$ are integers we have
$$d\left(\check f_{s_1}(q_1({\mathcal F})), \check f_{s_2}(q_1({\mathcal F}))\right)\leq 2\vert s_2-s_1\vert.$$
\end{proposition}

\begin{proof}

We will use Lemma \ref{le:distancefoliations}. Note that for every $s_1$, $s_2$ in $[2k,2k+1]$ we have
$$\begin{aligned} d\left(\check f_{s_1}(q_1({\mathcal F})), \check f_{s_2}(q_1({\mathcal F})\right)&= d \left( f_{mb}\circ \dots \circ f_{mq-k+1} (q_i^{s_1-2k}({\mathcal F})),  f_{mb}\circ \dots \circ f_{mb-k+1} (q_i^{s_2-2k}({\mathcal F}))\right)\\
&= d \left(q_i^{s_1-2k}({\mathcal F}),  q_i^{s_2-2k}({\mathcal F})\right)\\
&\leq 2, \end{aligned}$$
and similarly
for every $s_1$, $s_2$ in $[2k+1,2k+2]$ we have
$$\begin{aligned}d\left(\check f_{s_1}(q_1({\mathcal F})), \check f_{s_2}(q_1({\mathcal F})\right)&= d \left( f_{mb}\circ \dots \circ f_{mb-k+1} (q'{}_i^{s_1-2k-1}({\mathcal F})),  f_{mb}\circ \dots \circ f_{mb-k+1} (q'{}_i^{s_2-2k-1}({\mathcal F})\right)\\
&= d \left(q'{}_i^{s_1-2k-1}({\mathcal F}),  q'{}_i^{s_2-2k-1}({\mathcal F}))\right)\\
&\leq 2. \end{aligned}$$
We easily deduce Proposition  \ref{pr:distancefoliationsglobal}. 
\end{proof}

We immediately deduce

\begin{corollary} \label{co:distancefoliationstotal}It holds that
$$d\left(f^b(q_1({\mathcal F})),q_1({\mathcal F})\right)\leq 4mb.$$
\end{corollary}

Let us conclude this section by the following result

\begin{proposition} \label{pr:Brouwer} If $\tilde{\check I}=(\tilde{\check f}_s)_{s\in[0,2mb]}$ is the identity isotopy that lifts $\left(\check f_s{}_{\vert D_{a/b}^*}\right)_{s\in[0,2mb]}$, then every leaf $\tilde\phi$ of $\widetilde {q_1({\mathcal F})}$ is a Brouwer line of $\tilde{\check f}_{2mb}$.

\end{proposition}

\begin{proof} Let $\phi $ be a leaf of $\mathcal F$, then for every $s\in[0,2mq]$, we have
  
  $$\check f_s(q_1(\phi))= \begin{cases} f_{mq}\circ \dots \circ f_{mb-k+1} \circ q_{mb-k+1}^{s-2k}(\phi) & \text {if $s\in [2k,2k+1]$, $0\leq k<mb$}\\
   f_{mq}\circ \dots\circ f_{mb-k+1}  \circ q'{}_{mb-k+1}^{s-2k}(\phi)& \text {if $s\in [2k+1,2k+2]$, $0\leq k<mq$}\end{cases}$$
Equivalently, it holds that
   $$\check f_s(q_1(\phi)) = f_{mq}\circ \dots \circ f_{mb-k+1} \circ v_{mb-k+1}^{s-2k}(q_{mb-k+1}(\phi))$$
   if  $s\in [2k,2k+2]$, $0\leq k<mb$. 
   
   We deduce that for every leaf $\tilde \phi_1$ of $\widetilde {q_1({\mathcal F})}$, it holds that $L( \tilde{\check f}_{s_2}(\tilde \phi_1))\subset L( \tilde {\check f}_{s_1}(\tilde \phi_1))$ if $0\leq s_1<s_2\leq 2mb$ and that $\overline {L( \tilde{\check f}_{2k+2}(\tilde \phi_1))}\subset L(\tilde {\check f}_{2k}(\tilde \phi_1))$ if $0\leq k<mb$. In particular we have $\overline {L( \tilde{\check f}_{2mb}(\tilde \phi))}\subset L(\tilde \phi_1)$

\end{proof}

\section{Proof of Theorem \ref{th:principal}}

 \begin{proof}[Proof of Theorem \ref{th:principal}]

Let $f\in\mathrm{Diff}^1_{\omega}(\D)$ be an irrational  pseudo-rotation that coincides with a rotation on $\S$. We want to prove that $\mathrm{Cal}(f)=0$. We extended our map to the whole plane and decompose it in twist maps  as explained  is Section \ref{se:goodfoliation}. In particular we get a natural identity isotopy $I$ such that $\mathrm{rot}(I_{\vert \S})=\alpha$, where $\alpha\in\R\setminus \Q$ and we want to prove that  $\widetilde{\mathrm{Cal}}(I)=\pi^2\alpha$. We use the results of Section \ref{se:goodfoliation}, keeping the same notations. If $(a,b)\in\Z\times\N\setminus\{0\}$, satisfies $a\in(b\alpha,b\beta)$, then $f_{\vert D_{a/b}}$ is a piecewise diffeomorphism of class $C^1$. It is easy to see that $\widetilde{\mathrm{Cal}}(I_{\vert D_{a/b}})$ is well-defined, if one refers to the third definition given in the introduction. It can be explicitly computed, we have
$$\begin{aligned}\widetilde{\mathrm{Cal}}(I_{\vert D_{a/b}})&= \widetilde{\mathrm{Cal}} (I) + 2\int_{1}^{1+a/b-\alpha} (\pi r^2) (\alpha+r-1) 2\pi r\, dr\\
&= \widetilde{\mathrm{Cal}} (I) + 4\pi^2\int_{1}^{1+a/b-\alpha} r^3(\alpha+r-1) dr.\end{aligned}$$
Of course it holds that $$\lim_{{a/b}\to \alpha} \widetilde{\mathrm{Cal}}(I_{\vert D_{a/b}})=\widetilde{\mathrm{Cal}}(I_{\vert \D}).$$

We have constructed a radial foliation $\mathcal F$ on $\Delta_a^*\subset E_b$. We define ${\mathcal F}_1=q_1(\mathcal F)$ and denote $\tilde{\mathcal F}_1$ the lift of  ${\mathcal F}_1$ to $\tilde D_{a/b}$. We define $W_{a/b}=\{(z,z')\in D_{a/b}^*\,\vert\, z\not=z'\}$. Note that the area of $D_{a/b}$ is
$$A(a,b)= \pi(1+a/b-\alpha)^2$$and recall that the difference of the value of ${\bf h}$ between the points of $\Sigma_a$ and $\{0\}$ is
  $$C(a,b)=\pi(a-b \alpha) \left( 1 + (a/b -\alpha) +{(a/b-\alpha)^2\over 3}\right).$$

\begin{lemma}\label{le:firstbound}The following bound holds
$$\mu_{\omega}\times\mu_{\omega} \left(\{(z,z')\in W_{a/b}\,\vert\, \overline{\hat\tau}(z, z', {\mathcal F}_1, f^{-b}({\mathcal F}_1))\not=0 \} \right)\leq 2\,A(a,b)\,C(a,b).$$\end{lemma}

\begin{proof}
Fix $z\in  D_{a/b}^*$ and choose a lift $\tilde z\in \tilde D_{a,b}$. The set $$O(\tilde z) = \overline {L({\widetilde {\check f}_{2mq}{}^{-1}}( \tilde\phi_{\tilde z}))\cap R(\tilde \phi_{\tilde z}})$$ has measure $C(a,b)$ (for the lifted measure $\tilde\mu_{\omega}$) and projects onto a subset $O(z)\subset D_{a/b}$ satisfying $\mu_{\omega}(O(z))\leq C(a,b)$. Moreover, we have
$$\tilde z' \in \tilde O(\tilde z) \Longleftrightarrow  \tilde{\check  I}(\tilde z') \cap \tilde\phi_{\tilde z}\not=\emptyset, \enskip z' \in O(z) \Longleftrightarrow  \check I(z') \cap\phi_{z}\not=\emptyset. $$  
Note that 
$ \tilde {\check I}(\tilde z)$ and  $\tilde {\check I}(\tilde z')$ meet a common  leaf if and only if $\tilde z' \in \tilde O(\tilde z)$ or $\tilde z \in \tilde O(\tilde z')$
and that $  {\check I}(z)$ and ${\check I}(z')$ meet a common  leaf if and only if $z' \in O(z)$ or $z \in O(z').$
 Obviously, $I(z)$ and $I(z')$ meet a common leaf if $\overline{\hat\tau}(z, z', {\mathcal F}_1, f^{-b}({\mathcal F}_1))\not=0$ and so, it holds that
 $$\begin{aligned}{}& \enskip\enskip\enskip\mu_{\omega}\times\mu_{\omega}\left(\{(z,z')\in W_{a/b}\,\vert\, \hat\tau(z, z', {\mathcal F}_1, f^{-b}({\mathcal F}_1))\not=0\} \right)\\&\leq \int_{D_{a/b}^*} \mu_{\omega}( O(z)) \, d\mu_{\omega}(z)+ \int_{ D_{a/b}^*} \mu_{\omega}( O(z')) \, d\mu_{\omega}(z')\\
  &\leq 2A(a,b)\,C(a,b).\end{aligned}$$
 
\end{proof}

We can be more precise:

\begin{lemma}\label{le:secondbound}The following bound holds
$$\int_{W_{a/b}}\overline{ \hat\tau}(z, z', {\mathcal F}_1, f^{-b}({\mathcal F}_1)) \,d\mu_{\omega}(z) \,d\mu_{\omega} (z') \leq 8mb\,A(a,b)\,C(a,b).$$\end{lemma}

\begin{proof}We keep the same notations as in Lemma \ref{le:firstbound}. For every $(z,z')\in W_{a/b}$ define
 $$\nu(z,z')=\#\{k\in\Z\, \vert \, \overline{\hat\tau}(\tilde z, T^{-k}(\tilde z'), {\mathcal F}_1, f^{-b}({\mathcal F}_1))\not=0\}$$
where $\tilde z, \tilde z'$ are given lifts of $z,z'$.  Now, for every $z\in D_{a/b}^*$,  define a function $\gamma_z: D_{a/b}^*\to \N$ assigning to every point $z'\in D_{a/b}^*$ the number of its lifts that are in $O(\tilde z)$, where $\tilde z$ is a given lift of $z$. Suppose that $\gamma_z(z')\not=0$, or equivalently that $z'\in O(z)$. Then, for every $l\in\{1,\dots, \gamma_z(z')\}$ there exists a unique lift $\tilde z'\in \tilde O( \tilde z)$ of $z'$ such that $\tilde I(\tilde z')$ meets the leaves $T^r(\tilde \phi_{\tilde z})$, $0\leq r<l$. Note also that
$$\int_{ D_{a/b}^*} \gamma_z(z')\, d\mu_{\omega}(z')= \tilde\mu_{\omega}( \tilde O( \tilde z)) = C(a,b)$$
and that $$\nu(z,z')\leq \gamma_{z}(z')+\gamma_{z'}(z).$$
To conclude, we refer to Corollary \ref{co:distancefoliationstotal} that says that  for every $(\tilde z, \tilde z')\in W_{a/b}$, it holds that 
$\vert\hat\tau( \tilde z, \tilde z', {\mathcal F}_1, f^{-b}({\mathcal F}_1))\vert\leq 4mb$. \end{proof}

 Note also that for every couple $(a,b)$ defined as above, it holds that $\tilde{\check f}_{2mb}=(\tilde f_{\vert \tilde D_{a/b}})^b-T^a$ (or equivalently that $[\check I]=[I]_{\vert D_{a/b}}^b[T_{-a}]_{\vert D_{a/b}}$). Indeed both maps lift $f^b$ and coincide on $S_{a/b}$. 
If $\phi$ is a leaf of ${\mathcal F}_1$, then  $m_{\tilde{\check f}_{2mb}, \phi}$ is non negative and 
$$\int_{D^*_{a/b}} m_{\tilde{\check f}_{2mb}, \phi}\, d\mu_{\omega}=C(a,b).$$
We deduce that $\mu_{\omega}$ has a rotation vector (which was obviously known) and that  
$$\mathrm{rot}_{\tilde{\check f}_{2mb}}(\mu_{\omega})= \int_{D^*_{a/b}} m_{\tilde{\check f}_{2mb}, \phi}\, d\mu_{\omega}=C(a,b).$$
We deduce from Lemma \ref{le:secondbound} that $ \lambda_{f^b,{\mathcal F}_1}$ is integrable on $W_{a/b}$ for the measure $\mu_{\omega}\times\mu_{\omega}$  and that 
$$\left\vert \int_{W_{a,b}} \lambda_{f^b,{\mathcal F}_1} \,d\mu_{\omega}(z) \,d\mu_{\omega}(z')\right\vert \leq 8mb\,A(a,b)\,C(a,b).$$
According to the results of Section \ref{se:Calabi}, we deduce that $\mu_{\omega}$ has a self linking number for $\tilde{\check f}_{2mb}$ and that
$$\left\vert\mathrm{link}_{\tilde{\check f}_{2mb}} (\mu_{\omega}) -\int_{D_{a/b}^*}\mathrm{rot}_{\tilde{\check f}_{2mb}}(z) \,d\mu_{\omega}(z)\,d\mu_{\omega}(z')\right\vert \leq 8mb\,A(a,b)\,C(a,b).$$
But we know that $\tilde{\check f}_{2mb} =\check f^b\circ T^{-a}$ and so we obtain
$$\left\vert b\,\mathrm{link}_{ \tilde f_{\vert \tilde D_{a/b}}} (\mu_{\omega})-aA(a,b)^2- A(a,b)\int_{D_{a/b}^*}\mathrm{rot}_{\tilde{\check f}_{2mq}}(z) \, d\mu_{\omega}(z)\right\vert \leq 8mb\,A(a,b)\,C(a,b),$$
which can be written
$$\left\vert\mathrm{link}_{ \tilde f_{\vert \tilde D_{a/b}}}(\mu_{\omega})-{a\over b}A(a,b)^2 -{A(a,b)\,C(a,b)\over b}\right\vert \leq 8mA(a,b)C(a,b).$$
One can find a sequence $(a_n, b_n)_{n\geq 0}$ such that $\lim_{n\to+\infty} a_n- b_n\alpha =0$. Writing
$$\left\vert\mathrm{link}_{ \tilde f_{\vert \tilde D_{a_n/b_n}}}(\mu_{\omega})-{a_n\over b_n}A(a_n,b_n) ^2-{A(a_n,b_n)\,C(a_n,b_n)\over b}\right\vert \leq 8mA^2(a_n,b_n)\,C(a_n,b_n)$$ and letting $n$ tend to $+\infty$, one obtains
$$\mathrm{link}_{ \tilde f_{\vert \tilde \D}}(\mu_{\omega})-\pi^2\alpha=0.$$
But we know that $\mathrm{link}_{ \tilde f_{\vert \tilde \D}}(\mu_{\omega})= \widetilde{\mathrm{Cal}}(I)$, so we can conclude.
\end{proof}



\end{document}